\documentclass[11pt,reqno]{amsart}

\usepackage{lipsum}

\usepackage{graphicx}
\usepackage{array}
\usepackage{amssymb}
\usepackage{hyperref}
\usepackage{amsthm}
\usepackage{amsfonts}
\usepackage{amsmath}
\usepackage{bm}
\usepackage{mathrsfs}
\usepackage[all]{xy}
\usepackage{color}
\usepackage{subfigure}
\usepackage{tikz, tikz-cd}
\usepackage{enumerate}
\usepackage{anysize}
\usepackage{amscd}
\usepackage[letterpaper]{geometry}
\usepackage{multirow}
\usepackage{array}
\usepackage{booktabs}

\newtheorem{theorem}{Theorem}[section]
\newtheorem{proposition}[theorem]{Proposition}

\newtheorem{lemma}[theorem]{Lemma}

\theoremstyle{remark}

\theoremstyle{definition}
\newtheorem{definition}[theorem]{Definition}
\newtheorem{remark}[theorem]{Remark}

\numberwithin{equation}{section}
\numberwithin{figure}{section}
\numberwithin{table}{section}

\newcommand{\dC}{\mathbb{C}}

\newcommand{\dN}{\mathbb{N}}
\newcommand{\dR}{\mathbb{R}}

\renewcommand{\i}{\sqrt{-1}}

\newcommand{\dZ}{\mathbb{Z}}

\newcommand{\fs}{\mathfrak{s}}

\newcommand{\fM}{\mathfrak{M}}

\newcommand{\RR}{\mathbb{R}}
\newcommand{\ZZ}{\mathbb{Z}}
\newcommand{\PP}{\mathbb{P}}
\newcommand{\p}{\partial}

\newcommand{\cC}{\mathcal{C}}

\newcommand{\CC}{\mathbb{C}}

\DeclareMathOperator{\ALG}{ALG}

\DeclareMathOperator{\dvol}{dvol}

\DeclareMathOperator{\Id}{Id}
\DeclareMathOperator{\I}{I}
\DeclareMathOperator{\II}{II}
\DeclareMathOperator{\III}{III}
\DeclareMathOperator{\IV}{IV}

 \DeclareMathOperator{\model}{model}

\DeclareMathOperator{\SF}{sf}

\DeclareMathOperator{\Ima}{Im}

\DeclareMathOperator{\Rea}{Re}

\DeclareMathOperator{\Rm}{Rm}

\DeclareMathOperator{\Vol}{Vol}

\DeclareMathOperator{\Nil}{Nil}

\allowdisplaybreaks[3]

\begin{document}

 \title[Gravitational instantons with quadratic volume growth]{Gravitational instantons with quadratic volume growth}

 \author{Gao Chen} \thanks{The first named author is supported by the Project of Stable Support for Youth Team in Basic Research Field, Chinese Academy of Sciences, YSBR-001. The second named author is supported by NSF Grants DMS-1811096 and DMS-2105478. MSC 2020 Primary: 53C25, 53C26.
}
  \address{Institute of Geometry and Physics, University of Science and Technology of China, Shanghai, China, 201315}
 \email{chengao1@ustc.edu.cn}
 \author{Jeff Viaclovsky}
 \address{Department of Mathematics, University of California, Irvine, CA 92697}
 \email{jviaclov@uci.edu}

\begin{abstract}
There are two known classes of gravitational instantons with quadratic volume growth at infinity, known as type $\ALG$ and $\ALG^*$. Gravitational instantons of type $\ALG$ were previously classified by Chen-Chen. In this paper, we prove a classification theorem for ALG$^*$ gravitational instantons. We determine the topology and prove existence of ``uniform'' coordinates at infinity for both ALG and ALG$^*$ gravitational instantons. We also prove a result regarding the relationship between ALG gravitational instantons of order $\mathfrak{n}$ and those of order $2$. 
\end{abstract}

\date{}

\maketitle

\setcounter{tocdepth}{1}
\tableofcontents

\section{Introduction}
We begin with the following definitions.
\begin{definition}
\label{d:HK}
A hyperk\"ahler $4$-manifold $(X, g, I, J, K)$ is a Riemannian $4$-manifold $(X,g)$ with a triple of K\"ahler structures $(g, I), (g, J), (g, K)$ such that $IJ=K$. 
\end{definition}
We denote by $\bm{\omega} = (\omega_1, \omega_2, \omega_3)$ the K\"ahler forms associated to $I, J, K$, respectively. The form $\Omega = \omega_2 + \i \omega_3$ is a non-vanishing holomorphic $2$-form with respect to the complex structure $I$. These satisfy the condition
\begin{align}
\label{e:HKtriple}
\omega_1 \wedge \omega_1 = \frac{1}{2} \Omega \wedge \bar\Omega. 
\end{align}
Conversely, any 2-dimensional K\"ahler manifold $(X, g, I, \omega_1)$ with a non-vanishing holomorphic $2$-form $\Omega$ which satisfies \eqref{e:HKtriple} is hyperk\"ahler.
\begin{definition}
A \textit{gravitational instanton} $(X, g, \bm{\omega})$ is a complete hyperk\"ahler $4$-manifold $X$.
\end{definition}
If $g$ is any gravitational instanton satisfying a curvature decay condition 
\begin{equation}
|\Rm|=O(s^{-2-\epsilon}),\label{e:faster-decay}
\end{equation} 
for some $\epsilon > 0$ as $s \to \infty$, where $s(x) = d_g(x_0,x)$ for some $x_0 \in X$, Chen-Chen proved that it must be of type ALE, ALF, ALG, or ALH; see~\cite[Theorem~1.1]{CCI}. We refer the reader to \cite{BKN, CCII, CCIII, Hein, Kronheimer, Minerbe10} and references therein for more background on gravitational instantons. 

In this paper, we are interested in two classes of gravitational instantons known as type $\ALG$ and $\ALG^*$, whose (unique) tangent cones at infinity are two-dimensional flat cones. The $\ALG$ type satisfies the curvature decay \eqref{e:faster-decay}, but the $\ALG^*$ case has curvature decay $O(s^{-2}(\log s)^{-1})$, as $s \to \infty$. 
Both the $\ALG$ and $\ALG^*$ types have quadratic volume growth
\begin{align}
\label{qvg}
c \cdot t^2<\Vol(B_t(x_0))<C \cdot t^2
\end{align}
for $t$ sufficiently large.

For any rational elliptic surface $S$, and any singular fiber $D$ of finite monodromy, that is, of Kodaira type $\II, \III, \IV, \II^*, \III^*, \IV^*, $ or $\I_0^*$, Hein constructed in \cite{Hein} a nice background K\"ahler form on $S\setminus D$, which is $\partial \overline{\partial}$-cohomologous to the restriction of a K\"ahler form from $S$, and proved that the corresponding Tian-Yau metric \cite[Theorem~1.1]{TianYau} is an ALG gravitational instanton on $S \setminus D$. Conversely, it was shown in \cite{CCIII} 
that any ALG gravitational instanton can be compactified to a rational elliptic surface $S$, and the K\"ahler form with respect to the elliptic complex structure must be $\partial \overline{\partial}$-cohomologous to the restriction of a K\"ahler form from $S$. In this paper, we prove a similar classification for $\ALG^*$ gravitational instantons. For any rational elliptic surface $S$, and any singular fiber $D$ of Kodaira type $\I_{\nu}^*, \ 1 \leq \nu \leq 4$, in \cite{Hein}, Hein also constructed a nice background K\"ahler form on $S\setminus D$, which is $\partial \overline{\partial}$-cohomologous to the restriction of a K\"ahler form from $S$, and proved that the corresponding Tian-Yau metric \cite[Theorem~1.1]{TianYau} is asymptotic to a semi-flat metric near the divisor. 
We will show that these metrics are actually $\ALG^*$ gravitational instantons on $S \setminus D$ according to Definition~\ref{d:ALGstar}; see Section~\ref{s:model-metric}.
We also prove a converse; namely, that any ALG$^*$ gravitational instanton can be compactified to a rational elliptic surface $S$, and the K\"ahler form with respect to the elliptic complex structure must be $\partial \overline{\partial}$-cohomologous to the restriction of a K\"ahler form from $S$; see Theorem~\ref{t:ALGstar}. 

As a corollary of this classification result, we will determine the topology of ALG$^*$ gravitational instantons; see Theorem \ref{t:ALGstartopology}. Using Chen-Chen's classification of ALG gravitational instantons from \cite{CCIII}, we will also determine the topology of ALG gravitational instantons; see Theorem~\ref{t:ALGtopology}. In both cases, we will address the issue of the existence of ``uniform'' coordinates at infinity; see Theorem \ref{t:ALGstaruniform} and Theorem \ref{t:ALGuniform}. Finally, in Theorem \ref{t:order}, we prove a result about the relationship between ALG gravitational instantons of order $\mathfrak{n}$ and those of order $2$.

In the following subsections, we will outline the paper and state our main results in much more detail. 

\subsection{ALG$^*$ gravitational instantons}
\label{ss:ALGstar}
In Section~\ref{s:model-metric}, we will define the 
{\textit{standard $\ALG^*$ model space}}, which is denoted by 
\begin{align}
(\fM_{2\nu}(R),g^{\fM}_{\kappa_0,L},  \bm{\omega}_{\kappa_0,L}^{\fM}),
\end{align}
and depends on  parameters $\nu \in \ZZ_+$, $\kappa_0 \in \RR$ and $R, L \in \RR_+$.
Here, we just note that the manifold $\fM_{2\nu}(R)$ is diffeomorphic to $(R, \infty) \times \mathcal{I}_{\nu}^3$, 
where $\mathcal{I}_{\nu}^3$ is an infranilmanifold, which is a circle bundle of degree $\nu$ over a Klein bottle. We will let $r$ denote the coordinate on $(R, \infty)$, and define $\fs = r V^{1/2}$, where $V = \kappa_0 + \frac{\nu}{\pi} \log r$. The hyperk\"ahler structure is obtained via a Gibbons-Hawking ansatz, and explicit formulas for the metric and hyperk\"ahler structure can be found in Section~\ref{s:model-metric}. We note that the paramter $L$ is just an overall scaling parameter. 

\begin{definition}[ALG$^*$ gravitational instanton]
\label{d:ALGstar} 
A complete hyperk\"ahler $4$-manifold $(X,g,\bm{\omega} )$
is called an $\ALG^*$ gravitational instanton of order $\mathfrak{n} > 0$ with parameters 
$\nu \in \ZZ_+$, $\kappa_0 \in \RR$ and $L \in \RR_+$ if there exist an $\ALG^*$ model space $(\fM_{2\nu}(R),g^{\fM}_{\kappa_0,L}, \bm{\omega}^{\fM}_{\kappa_0,L})$ with $R >0$, a compact subset $X_R \subset X$, and a diffeomorphism 
$\Phi: \fM_{2\nu}(R) \rightarrow X \setminus X_R$ such that
  \begin{align}
 \big|\nabla_{g^{\fM}}^k(\Phi^*g- g^{\fM}_{\kappa_0,L})\big|_{g^{\fM}}&=O(\fs^{-k-\mathfrak{n}}), \\
 \big|\nabla_{g^{\fM}}^k(\Phi^*\omega_i- \omega_{i,\kappa_0,L}^{\fM})\big|_{g^{\fM}}&=O(\fs^{-k-\mathfrak{n}}), \ i = 1,2, 3,
\end{align}
as $\fs \to \infty$ for any $k\in\dN_0$. 
\end{definition}

\begin{remark}
We note that $\ALG^*$ gravitational instantons satisfy the following properties:
\begin{enumerate}

\item $\Vol(B_t(x_0)) \sim t^2$ as $t \to \infty$. 

\item The tangent cone at infinity is $\RR^2/ \{\pm 1\}$.

\item $|\Rm_g| = O( s^{-2} (\log s)^{-1})$ as $s \to \infty$. 

\item There exist $\ALG^*$ coordinates on $X$ so that the order satisfies $\mathfrak{n} \geq 2$; see \cite[Theorem~1.10]{CVZ2}.
\end{enumerate}
\end{remark}
The following is our main classification theorem for ALG$^*$ gravitational instantons. 

\renewcommand*{\thetheorem}{A}
\begin{theorem}[ALG$^*$ classification]  
We have the following relationship between $\ALG^*$ gravitational instantons and rational elliptic surfaces.  

\begin{enumerate}
\item 
Let $(X,g,\bm{\omega})$ be an $\ALG^*$ gravitational instanton with parameters $\nu, \kappa_0,$ and $L$. 
Then $\nu \leq 4$, and $X$ can be compactified to a rational elliptic surface $S$ with global section by adding a Kodaira singular fiber $D$ of type $\I_{\nu}^*$ at infinity, with respect to the complex structure $I$.  The $2$-form $\Omega=\omega_2+ \i\omega_3$ is a rational 2-form on $S$ with $\mbox{div}(\Omega) = -D$.
Furthermore, we can choose $S$ so that 
there exists a K\"ahler form $\omega$ on $S$ and a smooth function 
$\varphi: X \rightarrow \RR$ such that
\begin{equation}
\omega_1=\omega+\i\partial\bar\partial\varphi.
\end{equation}

\item Conversely, let $(S,I)$ be a rational elliptic surface with a type $\I_\nu^*$ fiber $D$.
For any $\kappa_0 \in \RR$, any K\"ahler form $\omega$ on $S$, and any rational 2-form $\Omega=\omega_2+\i\omega_3$ on $S$ with $\mbox{div}(\Omega) = -D$,  there exist $c > 0$, $L > 0$,  and a smooth function $\varphi: X \rightarrow \RR$, where $X \equiv S \setminus D$, such that 
\begin{equation}
(X, g, \omega_1=\omega+\i \partial\bar\partial\varphi, c \cdot \omega_2,c \cdot \omega_3)
\end{equation} is an $\ALG^*$ gravitational instanton with parameters $\nu, \kappa_0$, and $L$, where $g$ is the metric determined by $\omega_1$ and the elliptic complex structure $I$.
\end{enumerate}
\label{t:ALGstar} 
\end{theorem}
\setcounter{theorem}{0}
\renewcommand*{\thetheorem}{\thesection.\arabic{theorem}}
Part (1) solves a particular case of Yau's conjecture in \cite[page~246]{YauICM},
and shows how these hyperk\"ahler structures are related to the geometry of rational elliptic surfaces. We note that the construction in Part (2) is similar to the modification by Chen-Chen in \cite[Theorem~4.3]{CCIII} of Hein's construction \cite{Hein} using Tian-Yau's methods \cite[Theorem~1.1]{TianYau}.
Hein's construction produces hyperk\"ahler structures which are asymptotic to a semi-flat hyperk\"ahler structure. But, as mentioned above, our $\ALG^*$ model space is defined via a Gibbons-Hawking construction. An important aspect our proof is to show that the model semi-flat metrics are indeed Gibbons-Hawking; see Subsection \ref{ss:semi-flat}. 

An application of the above classification is to show that we can view any ALG$^*$ gravitational instanton with $1 \leq \nu \leq 4$ as living on a fixed manifold $X_{\nu}$.
\renewcommand*{\thetheorem}{B}
\begin{theorem} 
\label{t:ALGstartopology}
If $(X, g, \bm{\omega})$ and $(X',g', \bm{\omega'})$ are any two $\ALG^*$ gravitational instantons with $\nu = \nu'$, then $X$ is diffeomorphic to $X'$. 
\end{theorem}
\setcounter{theorem}{0}
\renewcommand*{\thetheorem}{\thesection.\arabic{theorem}}
Our next result is the following refinement of this, which takes into account the $\ALG^*$ coordinates. 
\renewcommand*{\thetheorem}{C}
\begin{theorem} 
\label{t:ALGstaruniform}
Let $(X,g,\bm{\omega})$ and $(X',g',\bm{\omega}')$ be $\ALG^*$ gravitational instantons of order~$2$ with the same $\nu, \kappa_0, L$, with
 $\ALG^*$ coordinates $\Phi:  \fM_{2\nu}(R)  \rightarrow  X \setminus X_R$
and  $\Phi':  \fM_{2\nu}(R) \rightarrow  X' \setminus X'_R$. Then there exists a mapping  $F :   \fM_{2\nu}(R) \rightarrow   \fM_{2\nu}(R)$ 
 and a diffeomorphism $\Psi' : X \rightarrow X'$ such that $F$ preserves the model hyperk\"ahler structure and $\Psi' \circ \Phi = \Phi' \circ F$ when restricted to $ \fM_{2\nu}(R)$ for a sufficiently large $R$. Consequently, the hyperk\"ahler structure $(X,(\Psi')^* g', (\Psi')^* \bm{\omega}')$ is $\ALG^*$ of order~$2$ in the $\ALG^*$ coordinates defined by $\Phi$.
\end{theorem}
\setcounter{theorem}{0}
\renewcommand*{\thetheorem}{\thesection.\arabic{theorem}}
In other words, we can view any ALG$^*$ gravitational instanton of order $2$ with parameters $\nu, \kappa_0$, and $L$  as a gravitational instanton on a \textit{fixed} manifold $X_{\nu}$ with respect to a \textit{fixed} ALG$^*$ coordinate system $\Phi_{X_\nu}$. Theorem~\ref{t:ALGstartopology} and Theorem~\ref{t:ALGstaruniform} will be proved in Section~\ref{s:Topology of gravitational instantons}. 

\subsection{ALG gravitational instantons}
\label{ss:ALG}
For background of analysis on ALG spaces, related classification results, and relations to moduli spaces of monopoles and Higgs bundles, 
we refer the reader to \cite{BB, BM, CCIII, CherkisKapustinALG, FMSW, Hein, HHM, Mazzeo} and also the references therein. 

In Definition~\ref{d:ALG-model} below, we will define the standard ALG model space
\begin{align}(\cC_{\beta,\tau,L}(R), g^{\cC},\bm{\omega}^{\cC}),
\end{align}
for parameters $L, R \in \RR_+$, $\beta \in \big\{\frac12, \frac16, \frac56, \frac14, \frac34, \frac13, \frac23\}$, and $\tau \in \mathbb{H}$, where $\mathbb{H} \subset \CC$ is the upper half-space. Except for the case that $\beta = 1/2$, the parameter $\tau$ is determined by $\beta$; 
see Table~\ref{ALGtable}. Here we just note that $\cC_{\beta,\tau,L}(R)$ is diffeomorphic to $(R, \infty) \times N_{\beta}^3$, where $N^3_{\beta}$ is a flat $3$-manifold which is a torus bundle over a circle, and the metric $g^{\cC}$ is a flat metric;  the explicit formulas are given in Subsection~\ref{ss:ALGmodel}. We let $r$ denote the coordinate on $(R, \infty)$.

\begin{definition}[ALG gravitational instanton]
\label{d:ALG-space} 
A complete hyperk\"ahler $4$-manifold $(X,g, \bm{\omega})$
is called an ALG gravitational instanton of order $\mathfrak{n} > 0$   
with parameters  $(\beta,\tau)$ as in Table~\ref{ALGtable}, and $L \in \RR_+$ if there exist $R>0$, a compact subset $X_R \subset X$, and a diffeomorphism 
$\Phi:   \cC_{\beta,\tau,L}(R) \rightarrow X \setminus X_R$,
such that
 \begin{align}
 \big|\nabla_{g^{\cC}}^k(\Phi^*g-g^{\cC})\big|_{g^{\cC}}&=O(r^{-k-\mathfrak{n}}), \\
 \big|\nabla_{g^{\cC}}^k(\Phi^*\omega_i- \omega_i^{\cC})\big|_{g^{\cC}}&=O(r^{-k-\mathfrak{n}}), \ i = 1,2,3, 
  \end{align}
as $r \to \infty$, for any $k\in\dN$.
\label{ALG-definition}
  \end{definition}

\begin{remark}
\label{r:ALG}
We note that $\ALG$ gravitational instantons satisfy the following properties:
\begin{enumerate}
\item $\Vol(B_t(x_0)) \sim t^2$ as $t \to \infty$. 

\item The tangent cone at infinity is a two dimensional cone with cone angle $2\pi\beta$.

\item The $\ALG$ order $\mathfrak{n}$ is at least $2$ in the $\I_0^*, \II, \III, \IV$ cases
and $ \mathfrak{n}=2-\frac{1}{\beta}$ in the $\II^*, \III^*, \IV^*$ cases; see~\cite[Theorem~A]{CCII}.
\end{enumerate}
\end{remark}
As a corollary of Chen-Chen's classification of ALG gravitational instantons in \cite{CCIII}, our next result shows that we can view any ALG  gravitational instanton with $\beta$ as in Table~\ref{ALGtable} as living on a fixed space $X_{\beta}$. 
\renewcommand*{\thetheorem}{D}
\begin{theorem}
\label{t:ALGtopology} If $(X, g, \bm{\omega})$ and $(X',g', \bm{\omega'})$ are any two $\ALG$ gravitational instantons with $\beta = \beta'$, then $X$ is diffeomorphic to $X'$. 
\end{theorem}
\setcounter{theorem}{0}
\renewcommand*{\thetheorem}{\thesection.\arabic{theorem}}
This will be proved in Section \ref{s:Topology of gravitational instantons}; see Theorem~\ref{t:Topology of gravitational instantons}.
Our next result is the following refinement of this, which takes into account the $\ALG$ coordinates. 
\renewcommand*{\thetheorem}{E}
\begin{theorem} 
\label{t:ALGuniform}
Let $(X,g,\bm{\omega})$ and $(X',g',\bm{\omega}')$ be $\ALG$ gravitational instantons of order~$\mathfrak{n}$ with the same $\beta, \tau, L$, with
 $\ALG$ coordinates $\Phi: \cC_{\beta, \tau, L}(R)   \rightarrow  X \setminus X_R$
and  $\Phi': \cC_{\beta,\tau,L}(R) \rightarrow  X' \setminus X'_R$. Then there exists a mapping $F :  \cC_{\beta,\tau,L}(R) \rightarrow  \cC_{\beta,\tau,L}(R)$
and a diffeomorphism $\Psi' : X \rightarrow X'$ such that $F$ preserves the model hyperk\"ahler structure and $\Psi' \circ \Phi = \Phi' \circ F$ when restricted to $\cC_{\beta, \tau, L}(R)$ for a sufficiently large $R$. Consequently, the hyperk\"ahler structure $(X,(\Psi')^* g', (\Psi')^* \bm{\omega}')$ is $\ALG$ of order~$\mathfrak{n}$ in the $\ALG$ coordinates defined by $\Phi$.
\end{theorem}
\setcounter{theorem}{0}
\renewcommand*{\thetheorem}{\thesection.\arabic{theorem}}
In other words, we can view any ALG gravitational instanton of order $\mathfrak{n}$ with parameters $\beta$, $\tau$, and $L$ as a gravitational instanton on a \textit{fixed} space $X_{\beta}$ with respect a \textit{fixed} ALG coordinate system~$\Phi_{X_\beta}$.

As mentioned above, when $\beta > 1/2$, the ALG order $\mathfrak{n}$ can be improved to $2 - \frac{1}{\beta}$. Our next results examines in this case the relationship between ALG gravitational instantons of order $2 - \frac{1}{\beta}$ and those of order $2$. Given an $\ALG$ gravitational instanton $(X_{\beta},g,\bm{\omega})$ with parameters $(\beta, \tau)$ as in Table~\ref{ALGtable} and $L >0$, we consider the periods 
\begin{align}
[\bm{\omega}] \in H^2_{dR}(X_{\beta}) \times H^2_{dR}(X_{\beta}) \times H^2_{dR}(X_{\beta}).
\end{align}
In Section \ref{s:order}, we will prove the following.
\renewcommand*{\thetheorem}{F}
\begin{theorem}
\label{t:order} Assume that $\beta > 1/2$, and let 
$(X_{\beta},g,\bm{\omega})$ be an $\ALG$ gravitational instanton with parameters $(\beta, \tau)$ as in Table~\ref{ALGtable} and $L >0$, which is of order $\mathfrak{n} > 0$ in the $\ALG$ coordinate system $\Phi$. Then there exists a diffeomorphism $F: \cC_{\beta,\tau,L}(R) \rightarrow \cC_{\beta,\tau,L}(R)$ homotopic to the identity, and a $2$-parameter family of $\ALG$ gravitational instantons  $(X_{\beta},g_{a,b},\bm{\omega}_{a,b})$ with $(a,b) \in \RR^2$  containing $(X_{\beta},g,\bm{\omega})$, all with the same parameters $\beta$, $\tau$, and $L$, and the same periods $[\bm{\omega}_{a,b}] = [\bm{\omega}]$, which are $\ALG$ of order $2 - \frac{1}{\beta}$ in the $\Phi \circ F$ coordinates when $(a,b) \neq (0,0)$, and $\ALG$ of order $2$ in the $\Phi \circ F$ coordinates when $(a,b) = (0,0)$. 
\end{theorem}
\setcounter{theorem}{0}
\renewcommand*{\thetheorem}{\thesection.\arabic{theorem}}
This result is an improvement of \cite[Theorem~1.7]{CCIII}.

\subsection{Acknowledgements} The authors would like to thank Hans-Joachim Hein, Rafe Mazzeo, Song Sun, and Ruobing Zhang for enlightening discussions.

\section{The model hyperk\"ahler structures}
\label{s:model-metric}
In this section, we explain some properties of $\ALG^*$ gravitational instantons in more detail. In Hein's thesis \cite{Hein}, he proved that by choosing a nice background metric, the corresponding Tian-Yau metric \cite[Theorem~1.1]{TianYau} is a gravitational instanton asymptotic to Greene-Shapere-Vafa-Yau's semi-flat ansatz in a neighborhood of a singular fiber of type $\I_{\nu}^*$ on a rational elliptic surface. One of the main goals of this section is to show that  Greene-Shapere-Vafa-Yau's semi-flat ansatz hyperk\"ahler structure is asymptotic to a Gibbons-Hawking ansatz hyperk\"ahler structure, thus the $\I_{\nu}^*$ metrics in~\cite{Hein} are $\ALG^*$ according to Definition~\ref{d:ALGstar}. We note this was also remarked in~\cite[page~197]{Foscolo2020}.

\subsection{Gibbons-Hawking construction}
\renewcommand{\t}[1]{\theta_{#1}}
In this subsection, we review the Gibbons-Hawking construction of the ALG$^*$ model metric. See \cite[Section~2]{CVZ2} for more details. Let $\nu$ be any positive integer.
Recall that the Heisenberg nilmanifold $\Nil^3_{2\nu}$ of degree $2\nu$ is the quotient of $\RR^3$ by the following actions
\begin{align}
\label{group1}
(\theta_1,\theta_2 , \theta_3) &\mapsto (\theta_1 + 2 \pi, \theta_2, \theta_3 ),\\
\label{group2}
 (\theta_1,\theta_2 , \theta_3)&\mapsto (\theta_1, \theta_2+ 2 \pi, \theta_3 + 2 \pi \theta_1),\\
\label{group3}
 (\theta_1,\theta_2 , \theta_3)&\mapsto ( \theta_1, \theta_2, \theta_3 +  2 \pi^2 \nu^{-1} ).
\end{align}
We consider
\begin{align}
\Theta \equiv \frac{\nu}{\pi} (d \theta_3 - \theta_2 d\theta_1), \quad V \equiv \kappa_0 + \frac{\nu}{\pi} \log r, \quad r\in(R,\infty),\ \kappa_0\in\dR, \ R>e^{\frac{\pi}{\nu}(1-\kappa_0)},
\end{align}
on the manifold
\begin{align}
S^1_{\theta_3} \to \widehat{\fM}_{2\nu}(R) \equiv (R, \infty) \times \Nil^3_{2\nu} \to \tilde{U} \equiv (\RR^2 \setminus \overline{B_{R}(0)}) \times S^1_{\theta_2},
\end{align}
where $(r,\theta_1)$ are polar coordinates on $\RR^2$.
The Gibbons-Hawking metric on $\widehat{\fM}_{2\nu}(R)$ is
\begin{align}
\label{metricexp}
g^{\widehat{\fM}}_{\kappa_0} = V ( dx^2 + dy^2 + d\theta_2^2) + V^{-1} \Theta^2 = V ( dr^2 + r^2 d\theta_1^2 + d \theta_2^2) + V^{-1} \frac{\nu^2}{\pi^2} \Big( d \theta_3 - \theta_2 d \theta_1 \Big)^2,
\end{align}
where $x + \i y \equiv r\cdot e^{\i \t1}$.

Consider an orthonormal basis
\begin{align}
\{ E^1, E^2, E^3, E^4 \} = \{ V^{1/2} dx, V^{1/2} dy , V^{1/2} d\t2, V^{-1/2} \Theta\}.
\end{align}
Then the model hyperk\"ahler forms on $\widehat{\fM}_{2\nu}(R)$ are given by
\begin{align}
\label{mkf1}
\omega_I&=\omega_{1, \kappa_0}^{\widehat{\fM}}  = E^1 \wedge E^2 + E^3 \wedge E^4= V dx \wedge dy + d \t2 \wedge \Theta,\\
\label{mkf2}
\omega_J&= \omega_{2, \kappa_0}^{\widehat{\fM}} = E^1 \wedge E^3 - E^2 \wedge E^4 = V dx \wedge d \t2 - dy \wedge \Theta,\\
\label{mkf3}
\omega_K&=\omega_{3, \kappa_0}^{\widehat{\fM}} = E^1 \wedge E^4 + E^2 \wedge E^3 = dx \wedge \Theta + V dy \wedge d \t2.
\end{align}
The $\ZZ_2$-action
\begin{align}
\label{nilaction}
\iota (r, \t1,\t2, \t3) = (r, \t1 + \pi, - \t2, - \t3)
\end{align}
induces an automorphism of the hyperk\"ahler structure, and
we define the $\ALG^*_{\nu}$ model space as
\begin{align}
(\fM_{2\nu}(R), g_{\kappa_0}^{\fM},\omega_{1,\kappa_0}^{\fM}, \omega_{2, \kappa_0}^{\fM}, \omega_{3,\kappa_0}^{\fM})
\equiv
(\widehat{\fM}_{2\nu}(R), g_{\kappa_0}^{\widehat{\fM}}, \omega_{1, \kappa_0}^{\widehat{\fM}}, \omega_{2, \kappa_0}^{\widehat{\fM}}, \omega_{3,\kappa_0}^{\widehat{\fM}})/\iota.
\end{align}
For any scaling parameter $L>0$, we define
\begin{align}
(\fM_{2\nu}(R), g_{\kappa_0, L}^{\fM}, \omega_{1, \kappa_0, L}^{\fM}, \omega_{2, \kappa_0, L}^{\fM}, \omega_{3,\kappa_0, L}^{\fM}) \equiv (\fM_{2\nu}(R), L^2 \cdot g_{\kappa_0}^{\fM}, L^2 \cdot \omega_{1, \kappa_0}^{\fM}, L^2 \cdot \omega_{2, \kappa_0}^{\fM}, L^2 \cdot \omega_{3,\kappa_0}^{\fM}).
\end{align}

\subsection{Semi-flat structure}
\label{ss:semi-flat}
In this subsection, we show that the ALG$^*$ model space has a Greene-Shapere-Vafa-Yau semi-flat structure \cite{GSVY, Hein}. 

We first explain Kodaira's model for an I$_{\nu}^*$ fiber. Assume that we have a local 
elliptic fibration $f: U \rightarrow B_\xi$, where $B_\xi$ is a disc with coordinate $\xi$ and the monodromy around the origin is of type I$_{\nu}^*$. Let $\pi : B_u \rightarrow B_\xi$ denote the double cover branched over the origin given by $\xi = u^2$. Then $f$ pulls back to an elliptic 
fibration $f': U' \rightarrow B_u$ with a singular fiber of type I$_{2 \nu}$ over the origin.
By \cite{Kodaira1963}, we can identify $f'$ over the punctured disc with 
\begin{align}
 (B_u^* \times \CC_v) /(\ZZ\oplus \ZZ),
\end{align}
where $\ZZ\oplus \ZZ$ acts by 
\begin{equation}
(m,n) \cdot (u,v)=(u,v+m\tau_{1}(u)+n\tau_{2}(u)),
\end{equation}
and $\tau_1(u) = 1, \tau_2(u) = - \i \frac{\nu}{\pi } \log u$.  
The $\ZZ_2$-action on $U'$ is given by $(u,v) \mapsto (-u, -v)$. 
\begin{remark}
\label{r:periods} 
If we use the holomorphic coordinates $(\xi,w) = (u^2,uv)$ on the quotient, 
then the periods are 
\begin{align}
\label{Istartau}
\tau_{1,\model}(\xi) = \xi^{1/2}, \quad \tau_{2,\model}(\xi) =  \frac{\nu}{2\pi \i} \xi^{1/2} \log (\xi).
\end{align}
  Furthermore, the monodromy matrix is given by 
\begin{align}
\label{Istarmono}
A_{\nu} = 
\left(
\begin{matrix}
-1  & - \nu \\
0 & -1 \\
\end{matrix}
\right).
\end{align}
\end{remark}
Let $\Omega' = g du \wedge dv$ be a meromorphic $2$-form on $U'$ away from the central fiber with 
$g(u) = u^{-2} k(u^2)$, where $k$ is a regular holomorphic function on $B_{\xi}$ with $k(0) \neq 0$. 
Clearly $\Omega'$ is invariant under the $\ZZ_2$-action, so descends to a $2$-form $\Omega$ on $U$.  

The semi-flat ansatz from \cite[Equation 3.14]{Hein} is given by 
\begin{align}
\label{heinsemi}
\omega_{\SF, \epsilon}' = \i |k(u^2)|^2 \frac{\nu |\log |u||}{\pi \epsilon} \frac{du \wedge d \bar{u}}{|u|^4} + \frac{\i}{2} \frac{\pi \epsilon}{\nu |\log|u||} (dv - \Gamma du)
\wedge ( d\bar{v} - \overline{\Gamma}d \bar{u}),
\end{align}
where $\Gamma(u,v)$ is a function given by 
\begin{align}
\Gamma(u,v) = \frac{1}{\i} \frac{ \Ima(v)}{u |\log|u||}. 
\end{align}
Note that this K\"ahler form is defined upstairs, but descends to the quotient space.

We choose real coordinates $(v_1,v_2)$ defined by 
\begin{align}
\label{vdef}
v = v_1\tau_1 + v_2 \tau_2 = v_1 + v_2 \frac{\nu}{\pi \i } \log u.
\end{align}
There is an $\RR$-action with period $\sqrt{\epsilon}$ given by
\begin{align}
(u, v_1, v_2) \mapsto (u, v_1 + \epsilon^{-1/2} t, v_2), 
\end{align}
which leaves $\omega_{\SF,\epsilon}'$ and $\Omega'$ invariant, and with associated vector field 
\begin{align}
Y =  \frac{1}{\sqrt{\epsilon}}  \frac{\p}{\p v_1} = \frac{1}{\sqrt{\epsilon}} \Big( \frac{\p}{\p v} +  \frac{\p}{\p \bar{v}} \Big).
\end{align}
This action is Hamiltonian for each of the three K\"ahler forms
\begin{align}
\omega_1 = \omega_{\SF,\epsilon}', \quad \omega_2 = \Rea(\Omega'), \quad \omega_3 = \Ima(\Omega'). 
\end{align}
We next compute the Hamiltonian functions $H_i$, that is, 
\begin{align}
\omega_i(Y, \cdot) = d H_i,
\end{align}
for $i = 1, 2, 3$. 
\begin{proposition} The moment map for the tri-Hamiltonian action associated to $Y$ 
\begin{align} 
\mu : U' \setminus (f')^{-1}(0) \rightarrow \RR \times \CC 
\end{align}
is given by 
\begin{align}
\mu = (H_1,H_2 + \i H_3) = \Big(\epsilon^{1/2} v_2 , - \epsilon^{-1/2} \int u^{-2} k(u^2) du \Big).
\end{align}
\end{proposition}
\begin{proof}
We first compute 
\begin{align}
\begin{split}
\omega_{\SF,\epsilon}'(Y, \cdot) &=  \frac{\i}{2} \frac{\pi \sqrt{\epsilon}}{\nu |\log|u||}  
\Big( d \bar{v} - \overline{\Gamma} d \bar{u} - dv  + \Gamma du \Big)\\
& =  \frac{\i}{2} \frac{\pi \sqrt{\epsilon}}{\nu |\log|u||}  
 \big(  2 \i  \Ima ( - dv + \Gamma du) \Big) 
=    \frac{\pi \sqrt{\epsilon}}{ \nu |\log|u||}  
 \big(   \Ima (  dv - \Gamma du) \big).
\end{split}
\end{align}
Differentiating \eqref{vdef} yields
\begin{align}
\label{e:dv}
d v = d v_1 + \Big( \frac{\nu}{\pi \i } \log u\Big) dv_2
+ \Big(v_2 \frac{\nu}{\pi \i u }  \Big) d u .
\end{align}
Noting that 
\begin{align}
\Ima(v) = \Ima \Big(  v_1 + v_2 \frac{\nu}{\pi \i } \log u\Big)
= - v_2 \frac{\nu}{\pi} \log|u|,
\end{align}
the term involving $\Gamma$ is given by
\begin{align}
- \Gamma du = -  \frac{1}{\i} \frac{ \Ima(v)}{u |\log|u||} du 
= \i v_2 \frac{\nu}{\pi u}  du.
\end{align}
Consequently, 
\begin{align}
\omega_{\SF,\epsilon}'(Y, \cdot) & =   \frac{\pi \sqrt{\epsilon}}{ \nu |\log|u||}  
 \Ima \Big( \frac{\nu}{\pi \i} \log(u) dv_2  \Big) = \sqrt{\epsilon} dv_2.
\end{align}
Therefore, we can choose  $H_1 = \epsilon^{1/2} v_2$. 

Next, we compute 
\begin{align}
\begin{split}
\Omega'(Y, \cdot) &=  u^{-2} k(u^2) du \wedge dv \frac{1}{\sqrt{\epsilon}}\Big( \frac{\p}{\p v} +  \frac{\p}{\p \bar{v}}, \cdot \Big)\\
&= -  \epsilon^{-1/2} u^{-2} k(u^2) du
= d \Big(  - \epsilon^{-1/2} \int u^{-2} k(u^2) du \Big),
\end{split}
\end{align}
so we can choose 
\begin{align}
H_2 + \i H_3 =  - \epsilon^{-1/2} \int u^{-2} k(u^2) du,
\end{align}
where $\int u^{-2}k(u^2) du$ is any $u$-antiderivative of $u^{-2} k(u^2)$. 

\end{proof}

\begin{proposition} 
\label{p:semi-flat}
The semi-flat metric is $\ALG^*$ near infinity, as defined in Definition~\ref{d:ALGstar}. 
\end{proposition}
\begin{proof}
The function $k$ admits a power series expansion in $\xi=u^2$. 
It is clear that the higher order terms will yield polynomially decaying terms in the following computation, so we just need to consider the case that $k$  is a constant. Multiplying $\Omega'$ by $e^{\i \theta}$ corresponds to a hyperk\"ahler rotation of $J,K$, so  
without loss of generality we may assume that $k =  ik_0$, where $k_0 \in \RR_+$. 
In this case, the moment map is 
\begin{align}
\mu = \Big( \epsilon^{1/2} v_2,  \i \epsilon^{-1/2}k_0 u^{-1} \Big).
\end{align}
Letting $I$ denote the elliptic complex structure, we next compute 
\begin{align}
\begin{split}
g(Y, Y) &= \omega_{\SF,\epsilon}'(Y, I Y)
=\i \epsilon^{-1} \omega_{\SF,\epsilon}'\Big(  \frac{\p}{\p v} +  \frac{\p}{\p \bar{v}}   , 
  \frac{\p}{\p v} - \frac{\p}{\p \bar{v}}\Big)\\
& = - 2 \i \epsilon^{-1}\omega_{\SF,\epsilon}'\Big(  \frac{\p}{\p v}   ,  \frac{\p}{\p \bar{v}}\Big)
= -2 \i \epsilon^{-1} \frac{\i}{2} \frac{\pi \epsilon}{\nu |\log|u||} 
=  \frac{\pi}{\nu |\log|u||}.
\end{split}
\end{align}
Next, define
\begin{align}
\label{Veqn}
V &\equiv g(Y, Y)^{-1} = \frac{ \nu}{\pi} |\log|u||, \\
\alpha_0(Z) &\equiv g(Y, Z).
\end{align}
We compute that
\begin{align}
 K^*\alpha_0(Z)&=\alpha_0(K Z)=g(Y, K Z)=-\Ima \Omega' (Y, Z)=
\Ima \Big\{ d\Big(\epsilon^{-1/2}\int u^{-2} \i k_0 du \Big) \Big\}(Z),\\
J^*\alpha_0(Z)&=\alpha_0(J Z)=g(Y, J Z)=-\Rea \Omega' (Y, Z)= \Rea \Big\{ d\Big( \epsilon^{-1/2}\int u^{-2} \i k_0 du \Big) \Big\} (Z),\\
I^*\alpha_0(Z)&=\alpha_0(I Z)=g(Y, I Z)=-\omega_{\SF,\epsilon}'(Y, Z)= - \epsilon^{1/2} d v_2 (Z),
\end{align}
so we have
\begin{align}
\alpha_1 & \equiv \Rea \Big( \epsilon^{-1/2} k_0 u^{-2} du\Big) = K^*\alpha_0, \\
\alpha_2 & \equiv \Ima \Big( \epsilon^{-1/2} k_0 u^{-2} du \Big) = - J^*\alpha_0,\\
\alpha_3 & \equiv  - \epsilon^{1/2}  d v_2 = I^*\alpha_0.
\end{align}
Then $V^{1/2}\alpha_1$, $V^{1/2}\alpha_2$, $V^{1/2}\alpha_3$, $V^{1/2}\alpha_0$ form the dual basis of an orthonormal basis. Define 
\begin{align}
\Theta \equiv V\alpha_0.
\end{align}
Then the hyperk\"ahler metric is 
\begin{align}
g = V ( \alpha_1^2 + \alpha_2^2 + \alpha_3^2) + V^{-1} \Theta^2,
\end{align}
and the K\"ahler forms are
\begin{align}
\omega_I & = V \alpha_1 \wedge \alpha_2 + \alpha_3 \wedge \Theta,\\
\omega_J & = V \alpha_1 \wedge \alpha_3 - \alpha_2 \wedge \Theta,\\
\omega_K & = \alpha_1 \wedge \Theta + V \alpha_2 \wedge \alpha_3.
\end{align}
We compute that 
\begin{align}
\alpha_0(Z) = \omega_{\SF,\epsilon}'(Y, I Z)
= - \omega_{\SF,\epsilon}'(I Y, Z)
= - \i \epsilon^{-1/2} \omega_{\SF,\epsilon}'\Big( \frac{\p}{\p v} - \frac{\p}{\p \bar{v}} , Z\Big).
\end{align}
Using \eqref{heinsemi}, this yields 
\begin{align}
\alpha_0  = \frac{\pi \epsilon^{1/2}}{ \nu |\log|u||}
\Rea ( dv - \Gamma du).
\end{align}
Using \eqref{e:dv} and \eqref{Veqn}, we arrive at 
\begin{align}
\Theta = V \alpha_0 = \epsilon^{1/2} \Big( dv_1 + \frac{\nu}{\pi} \arg(u) dv_2 \Big).
\end{align}
We also compute that 
\begin{align}
\alpha_1^2 + \alpha_2^2 = \frac{|k_0|^2}{2 \epsilon} |u|^{-4} ( du \otimes d \bar{u}
+ d \bar{u} \otimes du),
\end{align}
so we have 
\begin{align}
g &=  V \Big(   \frac{|k_0|^2}{2 \epsilon} |u|^{-4} ( du \otimes d \bar{u}
+ d \bar{u} \otimes du) 
+ \epsilon dv_2^2 \Big) + V^{-1} \epsilon \Big( dv_1 + \frac{\nu}{\pi} \arg(u) dv_2 \Big)^2\\
& =   V \Big(   \frac{|k_0|^2}{2 \epsilon} |u|^{-4} ( du \otimes d \bar{u}
+ d \bar{u} \otimes du) 
+ \epsilon dv_2^2 \Big) + \frac{\nu^2}{\pi^2} V^{-1} \epsilon \Big(  \frac{\pi}{\nu} dv_1 + \arg(u) dv_2 \Big)^2.
\end{align}

We consider the metric 
\begin{align}
\begin{split}
\tilde{g} = 4 \pi^2 \epsilon^{-1} g
&= V \Big(  4 \pi^2 \frac{|k_0|^2}{2 \epsilon^2} |u|^{-4} ( du \otimes d \bar{u}
+ d \bar{u} \otimes du)+  4 \pi^2 dv_2^2 \Big) \\
& \qquad + \frac{\nu^2}{\pi^2} V^{-1} \Big( \frac{2 \pi^2}{\nu} d( v_1) +  \arg(u) d(2 \pi v_2) \Big)^2.
\end{split}
\end{align}
We next show that this is isometric to the model metric given in Definition~\ref{d:ALGstar}. 
Observe that 
\begin{align}
4 \pi^2 \frac{|k_0|^2}{2 \epsilon^2} |u|^{-4} ( du \otimes d \bar{u}
+ d \bar{u} \otimes du)
\end{align}
is the Euclidean metric on $\RR^2$ in the coordinates $\zeta = 2 \pi \i k_0 \epsilon^{-1} u^{-1} = x + \i y$, 
so we define coordinates 
\begin{align}
r &= |\zeta| = 2 \pi k_0 \epsilon^{-1} |u|^{-1}, \\
\theta_1 & = \arg(\zeta) = - \arg(u) + \pi,\\
\theta_2 & = - 2 \pi v_2, \\
\theta_3 & = \frac{2 \pi^2}{\nu} v_1.
\end{align}
Then for $r$ sufficiently large, 
\begin{align}
V =  \frac{ \nu}{\pi} |\log|u|| = 
\frac{ \nu}{\pi} \Big( \log r - \log ( 2 \pi k_0 ) + \log \epsilon \Big)
\equiv \frac{ \nu}{\pi} \log r  + \kappa_0,
\end{align}
and the rescaled metric $\tilde{g}$ is then 
\begin{align}
\tilde{g} = V (  dr^2 + r^2 d\theta_1^2 + d\theta_2^2)
+  \frac{\nu^2}{\pi^2} V^{-1}  
\Big( d \theta_3 + (\theta_1 - \pi) d \theta_2 \Big)^2.
\end{align}
The mapping
\begin{align}
\tilde{\theta}_3 = \theta_3 + \theta_1 \theta_2 - \pi \theta_2
\end{align}
defines a coordinate change $H$ satisfying 
\begin{align}
H^* \Theta = d \tilde{\theta}_3 - \theta_2 d \theta_1 \equiv \tilde{\Theta},
\end{align}
so we have
\begin{align}
H^* \tilde{g} &= H^* (4 \pi^2 \epsilon^{-1} g) = V (  dr^2 + r^2 d\theta_1^2 + d\theta_2^2)
+  \frac{\nu^2}{\pi^2} V^{-1} \tilde{\Theta}^2, \\
H^*\tilde{\omega}_I &= H^* (4 \pi^2 \epsilon^{-1} \omega_I) = V d x \wedge d y + d \t2 \wedge \tilde{\Theta}, \\
H^* \tilde{\omega}_J &= H^* (4 \pi^2 \epsilon^{-1} \omega_J) = V d x \wedge d \t2 - dy \wedge \tilde{\Theta}, \\
H^*\tilde{\omega}_K &= H^* (4 \pi^2 \epsilon^{-1} \omega_K) = d x \wedge \tilde{\Theta} + V d y \wedge d \t2,
\end{align}
which are exactly \eqref{metricexp}, \eqref{mkf1}, \eqref{mkf2}, and \eqref{mkf3}.  
\end{proof}
Note that if we reverse the above construction, we see that the leading terms of any ALG$^*$ gravitational instanton admit a semi-flat structure.

\begin{remark}
\label{r:parameters} 
It is easy to see from the above proof that the parameters $\epsilon, i k_0$ in the semi-flat metric are related to the paramters $\kappa_0, L$ in the Gibbons-Hawking construction by 
\begin{align}
4\pi L^2 = \epsilon, \quad \kappa_0 = \frac{\nu}{2\pi} \log \Big( \frac{\epsilon}{2 \pi k_0} \Big). 
\end{align} 
\end{remark}

\subsection{ALG model space}
\label{ss:ALGmodel}
In the ALG case, we have the following definition of the model space. 
 \begin{definition}[Standard ALG model]\label{d:ALG-model}
  Let $\beta\in(0,1]$, $\tau\in\mathbb{H}\equiv\{\tau\in\dC|\Ima\tau>0\}$ be the parameters in Table~\ref{ALGtable}, and  $L \in \RR_+$.
 Let $\cC_{\beta,\tau,L}$ be the  manifold obtained by identifying $(\mathscr{U},\mathscr{V})$ with $(e^{\i\cdot 2\pi  \beta}\mathscr{U},e^{- \i  \cdot 2\pi\beta}\mathscr{V})$ in the space \begin{equation}\{(\mathscr{U},\mathscr{V}) \ | \ \arg \mathscr{U}\in[0,2\pi\beta]\}\subset(\mathbb{C}\times \mathbb{C})/(\dZ\oplus \dZ),
\end{equation}
where $\dZ\oplus \dZ$ acts on $\mathbb{C}\times \mathbb{C}$ by
\begin{equation}
(m,n)\cdot (\mathscr{U},\mathscr{V})= \Big(\mathscr{U}, \mathscr{V}
 + (m+n\tau) \cdot L \Big),
\ (m,n)\in\dZ\oplus \dZ.
\end{equation}
Define 
\begin{equation}
\cC_{\beta,\tau,L}(R) \equiv \{|\mathscr{U}|>R\} \subset \cC_{\beta,\tau,L}.
\end{equation}
Then there is a flat hyperk\"ahler metric 
\begin{equation}
   g^{\cC}=\frac{1}{2}(d\mathscr{U}\otimes d\mathscr{\bar U} + d\mathscr{\bar U} \otimes d \mathscr{U}+d\mathscr{V}\otimes d\mathscr{\bar V} + d\mathscr{\bar V} \otimes d \mathscr{V})
  \end{equation}
 on $\cC_{\beta,\tau,L}(R)$ with K\"ahler form
\begin{equation}
   \omega_1^{\cC}=\frac{\i}{2}(d\mathscr{U}\wedge d\mathscr{\bar U}+d \mathscr{V}\wedge d\mathscr{\bar V}),
  \end{equation}
 and holomorphic $2$-form
\begin{equation}
  \Omega^{\cC}= \omega_2^{\cC} + \i \omega_3^{\cC} =d\mathscr{U}\wedge d \mathscr{V}.
 \end{equation}
Each flat space $(\cC_{\beta,\tau,L}(R), g^{\cC},\bm{\omega}^{\cC})$ given as the above  is called a {\it standard $\ALG$ model}.
\end{definition}

\begin{table}[h]
\caption{Invariants of ALG spaces.}
\label{ALGtable}
 \renewcommand\arraystretch{1.5}
\begin{tabular}{|c|c|c|c|c|c|c|c|c|} \hline
$\infty$ & $\I_0^*$ & $\II$ & $\II^*$ & $\III$ & $\III^*$ & $\IV$ & $\IV^*$\\\hline
$\beta\in(0,1]$ &  $\frac{1}{2}$  & $\frac{1}{6}$ & $\frac{5}{6}$ & $\frac{1}{4}$ & $\frac{3}{4}$ & $\frac{1}{3}$ & $\frac{2}{3}$\\ [5pt]\hline
  $\tau\in\mathbb{H}$ & Any & $e^{\i \cdot \frac{2\pi}{3}}$ & $e^{\i \cdot \frac{2\pi}{3}}$ & $\i$ & $\i$ & $e^{\i \cdot \frac{2\pi}{3}}$ & $e^{\i \cdot \frac{2\pi}{3}}$ \\\hline
\end{tabular}
\end{table}
\begin{remark}
\label{r:IJK}
We denote the complex structures acting on tangent vectors associated to 
$\omega_1^{\cC}, \omega_2^{\cC}, \omega_3^{\cC}$ by $I_{\cC}, J_{\cC}, K_{\cC}$, respectively.
Let $I_{\cC}^*, J_{\cC}^*, K_{\cC}^*$ be the dual mappings acting on 1-forms. Then we have
\begin{align}
I^*_{\cC} (d \mathscr{U}) &= \i d \mathscr{U}, &I^*_{\cC} (d \mathscr{V}) &= \i d \mathscr{V},\\
J^*_{\cC} (d \mathscr{U}) &= - d \mathscr{\bar V}, &J^*_{\cC} (d \mathscr{V}) &= d \mathscr{\bar U}, \\
K^*_{\cC} (d \mathscr{U}) &= - \i d \mathscr{\bar V}, &K^*_{\cC} (d \mathscr{V}) &= \i d \mathscr{\bar U}.
\end{align}
These will be used below in Section \ref{s:order}.
\end{remark}

\begin{remark} The model space has the following properties. Letting $r = |\mathscr{U}|$,  the cross-section $r = r_0$ is a \textit{flat} $3$-manifold.  
There is a holomorphic map $z_{\cC}:\cC_{\beta,\tau,L}(R)\to\mathbb{C}$ defined as $z_{\cC} = \mathscr{U}^{\frac{1}{\beta}}$, with torus fibers which have area $L^2 \cdot \Ima \tau$.  The infinite end of the model space compactifies 
complex analytically by adding a singular fiber of the specified type in the first row of Table \ref{ALGtable}.
\end{remark}

\section{Classification of  ALG$^*$ gravitational instantons}
\label{s:ALGstarclass}
In this section, we will prove Theorem \ref{t:ALGstar}. In Subsection \ref{ss:acALG}, we will study the compactification of ALG$^*$ gravitational instantons. Then in Subsection \ref{ss:coALG}, we will classify ALG$^*$ gravitational instantons.
\subsection{Compactification of ALG$^*$ gravitational instantons}
\label{ss:acALG}

\begin{proposition}
  \label{p:fircom}
Using the complex structure $I$, any $\ALG^*$ gravitational instanton $(X, g, I, J, K)$ of order $2$ must be biholomorphic to an elliptic surface $S$ minus an $\I_{\nu}^*$-fiber.
\end{proposition}
\begin{proof}
Since $(r e^{\i\t1})^2$ is invariant under $\iota$, 
by \cite[Theorem~1.3]{CVZ2}, there exists a harmonic function $z$ on $X$ such that $\Phi^* z$ is asymptotic to $(r e^{\i\t1})^2$ on $\fM_{2\nu}(R)$.
Since 
\begin{equation}
I_{\fM}^* d ( (r e^{\i\t1})^2 ) = \i d ( (r e^{\i\t1})^2 ),
\end{equation}
we have 
\begin{equation}
\eta \equiv I_{X}^* d z - \i d z = O(s^{-1 + \mu})
\end{equation}
for all $\mu>0$, and $\Delta_{X}\eta = 0$. 
By Bochner's formula and the maximum principle, $\eta = 0$ on $X$, which implies that $z$ is holomorphic on $X$.

Since $ d z - d ( (\Phi^{-1})^* (r e^{\i\t1})^2 ) = O(s^{-1 + \mu})$, the fiber at $z$ on $X$ for $|z|$ large is diffeomorphic to the fiber at $(r e^{\i\t1})^2$ on $\fM_{2\nu}(R)$, which is the same as the fiber at $r e^{\i\t1}$ on $\widehat{\fM}_{2\nu}(R)$ for $r$ large. This implies that the fibration $z: X \to \CC$ is elliptic.

  We next prove the existence of a holomorphic section $\sigma_1$ defined on a neighborhood of infinity using a similar argument as in \cite[Section~4.7]{CCI}. 
First, we define a $\overline{\p}$-operator acting on sections. 
Above we have proved there is a holomorphic function
$z: X \rightarrow \CC$.
Let $U =  \CC \setminus \overline{B_{R^2}(0)}$ and let $\sigma : U \rightarrow X$ be any smooth section, that is $z \circ \sigma = \Id_U$. The differential of $\sigma : U \rightarrow  X$ is
\begin{align}
  \sigma_* : TU \rightarrow T X,
\end{align}
where these are the real tangent bundles. If we complexify, then
\begin{align}
\sigma_* \otimes \CC : T^{1,0}(U) \oplus T^{0,1}(U)
 \rightarrow T^{1,0}(X) \oplus T^{0,1}(X).
\end{align}
Letting $\Pi : T_{\CC}X \rightarrow T^{1,0}(X)$ denote the projection mapping, 
we have a mapping
\begin{align}
  \Pi \circ ( \sigma_* \otimes \CC) :  T^{0,1}(U)
  \rightarrow  T^{1,0}(X).
\end{align}
Given any point $p \in U$, $z$ gives a local holomorphic coordinate on the base, and we choose a local affine coordinate $w$ on the fiber such that $w= 0$ is a locally defined holomorphic section and on each fiber, $w$ can be viewed as the coordinate on $\CC$ if we write the fiber as $\CC/\ZZ^2$.
Then locally we can write $\sigma : z \mapsto (z, s(z, \bar{z}))$
for a function $z : U \rightarrow \CC$, 
and we see that 
\begin{align}
  \Pi \circ (\sigma_* \otimes \CC) \Big( \frac{\p}{\p \bar{z}} \Big)
  = \frac{\p s}{\p \bar{z}}(z,\bar{z}) \frac{\p}{\p w} .
\end{align}
In other words, we have
\begin{align}
  \Pi \circ (\sigma_* \otimes \CC) :  T^{0,1}(U)  \rightarrow  T^{1,0}(F),
\end{align}
where $T^{1,0}(F)$ is the $(1,0)$ part of the vertical tangent bundle. 
This defines an operator from smooth sections over $U$ 
\begin{align}
  \overline{\p} : \Gamma(U) \rightarrow \Gamma(
  \Lambda^{0,1}(U) \otimes \mathbb{L}^{-1})
\end{align}
such that $\overline{\p} \sigma = 0$ if and only if $\sigma$ is a holomorphic section,
where $\mathbb{L}^{-1} \equiv  \sigma^* T^{1,0}(F)$.

In coordinates, since each fiber is a compact Riemann surface,  
we have that the holomorphic $2$-form is
$\Omega = f(z) d z \wedge d w$,
 where $f: U \rightarrow \CC$ is non-vanishing. 
By plugging in the $T^{1,0}(F)$ component to $\Omega$, we can define
\begin{align}
  \Omega \circ \overline{\p} : \Gamma(U) \rightarrow C^{\infty}(U, \CC),
\end{align}
since $\Lambda^{1,0}(U) \otimes \Lambda^{0,1}(U) \cong \Lambda^{1,1}(U)$ is a trivial bundle. In coordinates, this mapping is given by
\begin{align}
  \Omega \circ \overline{\p} \sigma = f(z)  \frac{\p s}{\p \bar{z}}(z,\bar{z}).
\end{align}

Next, recall that on the model space,  the zero section of the semi-flat structure is holomorphic. Since $z$ is asymptotic to $(\Phi^{-1})^*(r e^{\i\t1})^2$, we see that the zero section intersects each far-enough fiber of $z$ at exactly one point. This defines an ``almost holomorphic'' section $\sigma_0 : U \rightarrow X$ near infinity.  Locally we can write $\sigma_0 (z) = (z, s_0(z,\bar{z}))$, and we define 
\begin{align}
e(z, \bar{z}) =  f(z)  \frac{\p s_0}{\p \bar{z}}(z,\bar{z}).
\end{align}
Since $U$ is biholomorphic to a punctured disc $\Delta^*$ and 
$H^{1}(\Delta^*, \mathcal{O}) = 0$,  
there exists $E:  U \rightarrow \CC$ solving
\begin{align}
 \frac{\p}{\p \bar{z}} E ( z , \bar{z}) = e(z, \bar{z}).
\end{align}
We then define 
\begin{align}
s_1(z, \bar z) = s_0(z, \bar z) - E(z ,\bar z)/f(z),
\end{align}
and use this to define a section $\sigma_1$ over $U$. We then have
\begin{align}
\begin{split}
\frac{\partial}{\partial \bar z} ( s_0 (z, \bar z) - E(z,\bar z)/f(z) ) 
&= \frac{\partial}{\partial \bar z} ( s_0 (z, \bar z))
- \frac{\partial}{\partial \bar z} (  E(z,\bar z))/f(z)) \\
&= \frac{\partial}{\partial \bar z} ( s_0 (z, \bar z))
- (1/f(z)) f(z)  \frac{\partial}{\partial \bar z} ( s_0 (z, \bar z)) = 0.
\end{split}
\end{align}
Consequently, $\sigma_1$ is a holomorphic section over $U$.  

We then use the following fact from \cite{Kodaira1963}: any elliptic surface with a section defined over a punctured disc $\Delta^*$ is biholomorphic to  
$(\Delta^* \times \CC) / (\ZZ \tau_1 \oplus \ZZ \tau_2)$, where $\tau_1, \tau_2$ are the periods. Recalling Remark~\ref{r:periods}, there exists a local coordinate $u$ on the base such that 
the periods of the model space are given by $\tau_{1, \model}(u), \tau_{2,\model}(u)$. 
On the ALG$^*$ manifold, $\tau_1$ and $\tau_2$ are asymptotic to the model periods in ALG$^*$ coordinates. Then by a holomorphic coordinate transformation of the base, we can find coordinates $\tilde{u}$ asymptotic to $z^{-1}$ so that $\tau_1(\tilde{u}) = \tau_{1,\model}(\tilde{u})$ and $\tau_2(\tilde{u}) = \tau_{2,\model}(\tilde{u})$. 
This shows that the ALG$^*$ manifold is biholomorphic to the model space near infinity. Finally, the argument of Case (v), Sub-case (2) in \cite[Section~8]{Kodaira1963} shows that we can add a central fiber of type $\I_{\nu}^*$ to compactify and obtain $S$.  
\end{proof}

\begin{proposition}
Using the complex structure $I$, any $\ALG^*$ gravitational instanton $(X, g, I, J, K)$ of order $2$ is biholomorphic to a rational elliptic surface $S$ with global section, minus an $\I_{\nu}^*$ fiber $D$, with $\nu \in\{1, 2, 3, 4\}$.
\label{t:Compactification}
\end{proposition}
\begin{proof} We first show that the elliptic surface constructed in Proposition \ref{p:fircom} is a rational elliptic surface. To see this, it follows from \cite[Corollary~1.6]{CVZ2} that the first betti number $b_1(X)=0$. 
Since $D$ is a configuration of $2$-spheres corresponding to an extended Dynkin diagram of dihedral type, we have $b_1(D) = 0$.  Let $N$ be an open neighbohood of $D$ which deformation retracts onto $D$ and such that $X \cap N$ is connected. The following portion of the reduced Mayer-Vietoris sequence 
\begin{equation}
\begin{tikzcd}
 0 \arrow[r] &  H^{1}(S; \RR) \arrow[r] &  H^{1}(X;\RR)   \oplus H^{1}(N;\RR) 
\end{tikzcd}
\end{equation}
implies that $b_1(S) = 0$. Next, it is straightforward to see that $\Omega = \omega_2 + \i \omega_3$ defines a meromorphic $2$-form on $S$ with a pole of order $1$ along the divisor $D$ at infinity. Consequently, $\mbox{div}(\Omega) = -D$, which implies that $-K = [D]$. Similar to the proof of \cite[Theorem~3.3]{CCIII}, by Kodaira’s classification of complex surfaces and the Castelnuovo theorem, $S$ must be biholomorphic to $\mathbb{P}^2$ blown up at 9 points, so there exists an exceptional curve $E$ satisfying $E^2=-1$. By the adjunction formula, $E \cdot K=-1$, so $E \cdot D=1$. Let $D'$ be any singular fiber in the interior of $X$. Then $D'$ is homologous to $D$ in $S$. If $D'$ had multiplicity, say $D'=mC$, then $E \cdot C=1/m$ which implies that $m=1$, so that there are no multiple fibers. This implies that there exists a global holomorphic section of the elliptic fibration; see for example \cite{HarbourneLang}.
Finally, by the classification of singular fibers on rational elliptic surfaces due to Miranda-Persson, $\nu\in\{1,2,3,4\}$; see \cite{Miranda1990}. 
\end{proof}
\subsection{Classification of ALG$^*$ gravitational instantons}
\label{ss:coALG}
In this subsection, we complete the proof of Theorem~\ref{t:ALGstar}. 
Recall that we have proved in Proposition \ref{p:fircom} that the holomorphic function $z$ provides an elliptic fibration with associated holomorphic section $\sigma_1$ over $\{R^2<|z|<\infty\}$ which provides a compactification. 
Choose an arbitrary K\"ahler form $\omega_{S}$ on the compactification such that the area of each fiber is the same as $\omega_{\ALG^*}$. Denote this area by $\epsilon$. 
\begin{proposition}
\label{p:sphi}
There exist a holomorphic section $\sigma_2$ and a smooth real-valued function $\varphi_1$, both defined over $\{\infty>|z|>R^2\} \subset X$, such that 
\begin{equation}
\omega_{S}=\omega_{\SF, \epsilon}[\sigma_2]+ \i \partial\bar\partial\varphi_{1}
\end{equation}
on $\{\infty>|z|>R^2\}$, where $\omega_{\SF, \epsilon}[\sigma_2]$ is Greene-Shapere-Vafa-Yau's semi-flat metric \cite{GSVY} with $\sigma_2$ as the zero section, $\Omega$ as the holomorphic 2-form, and area of each fiber given by  $\epsilon$.
\end{proposition}
\begin{proof}
This is proved in Claim~1 on page 382 of \cite{Hein}, which is based on \cite[Lemma~4.3]{GW}.
\end{proof}
Similarly, there exists a holomorphic section $\sigma_3$ and a smooth function $\varphi_2$, both defined over $\{\infty>|z|>R^2\}$, such that 
\begin{equation}
\omega_{\ALG^*}=\omega_{\SF,\epsilon}[\sigma_3]+ \i \partial\bar\partial\varphi_{2}
\end{equation}
on $\{\infty>|z|>R^2\}$. Let $T$ be the translation by adding $\sigma_2-\sigma_3$ on each fiber, which satisfies $T^*\Omega = \Omega$ because $\sigma_2$ and $\sigma_3$ are both holomorphic. Then 
\begin{equation}\omega_{\ALG^*}=T^*(\omega_{S}-\i\partial\bar\partial\varphi_{1})+\i\partial\bar\partial\varphi_{2}.\end{equation}
If we use the local holomorphic section $\sigma_1-\sigma_2+\sigma_3$ to compactify $X$ into a rational elliptic surface $S'$, then 
\begin{equation}
T^*(\omega_{S})=\omega_{\ALG^*}+\i\partial\bar\partial\varphi_{3}
\end{equation}
can be extended to a smooth K\"ahler form on $\{\infty>|z|>R^2\}\subset S'$, where 
\begin{equation}
\varphi_{3}=T^*\varphi_{1}-\varphi_{2}
\end{equation}
is a smooth function.
To obtain a K\"ahler form on $S'$, we choose a cut-off function $\chi$ on $\RR$ which is 1 on $(2,\infty)$ and is 0 on $(-\infty,1)$. We also choose a smooth non-negative function $f$ on $\CC$ which is 1 on $R^2 \le |z| \leq 4R^2$ and is 0 on $|z|<R^2/4$ and $|z|>16R^2$. We can find a smooth real-valued function $\psi$ on $\CC$ such that 
\begin{equation}
\i \partial\bar\partial \psi = \i f  dz\wedge d\bar z.
\end{equation} 
Then
\begin{equation}
\omega=\omega_{\ALG^*} + \i \partial\bar\partial \Big( \chi \Big(\frac{\sqrt{|z|}}{R}\Big)\varphi_{3}+ C  \cdot \psi(z) \Big)
\end{equation}
will be the required smooth K\"ahler form on $S'$ for $C$ sufficiently large. This finishes the proof of Part (1). 

For Part (2), let $z : S \rightarrow \PP^1$ be the elliptic projection such that $D$ lies over $[0,1]$. Then we can view $z: X \rightarrow \CC$, where $X = S \setminus D$.  
By Remark \ref{r:parameters}, there exists $c > 0$ and $L > 0$ such that for any holomorphic section $\sigma$ defined over $\{R^2<|z|<\infty\}$, the Greene-Shapere-Vafa-Yau's semi-flat metric $\omega_{\SF, \epsilon}[\sigma]$ \cite{GSVY} using $\sigma$ as the zero section, $c \cdot \Omega$ as the holomorphic 2-form,  and using the same area $\epsilon$ on each fiber as $\omega$, has an ALG$^*$ end with parameters $\nu$, $\kappa_0$ and $L$. The function $\Phi^* z$ is asymptotic to $c' (r e^{\i \t1})^2$, where $c' \in \CC \setminus \{0\}$. 
Similar to Proposition \ref{p:sphi},
there exists a holomorphic section $\sigma$ 
and a smooth function $\varphi_4$, both defined over $\{R^2<|z|<\infty\}$, such that 
\begin{equation}
\omega=\omega_{\SF, \epsilon}[\sigma]+\i \partial\bar\partial\varphi_{4}.
\end{equation}
Then we use the following modification as in \cite[Theorem~4.3]{CCIII} instead of the original construction in \cite{Hein}. 
Choose $t$ sufficiently large so that 
\begin{equation}
\omega_{t}\equiv \omega - \i\partial\bar\partial \Big( \chi \Big(\frac{\sqrt{|z|}}{R}\Big) \varphi_{4} - t \cdot \psi (z) \Big)
\end{equation}
is positive, and the integral
\begin{equation}
\int_{X} \Big(\omega_{t}^2-\frac{c^2}{2}\Omega\wedge\bar\Omega\Big)
\end{equation}
is positive. The $d z \wedge d \bar z$ component of $\omega_t$ is asymptotic to 
\begin{equation}
\frac{\i}{8 |c'|}\Big(\kappa_0  + \frac{\nu}{2\pi} \log \Big|\frac{z}{c'}\Big|\Big) |z|^{-1} d z\wedge d \bar z.
\end{equation}
If we fix a large enough $R$, then
\begin{equation}
\omega_{t, t'} \equiv \omega_{t}-\chi\Big(\frac{\sqrt{|z|}}{R}-100\Big)\Big(1-\chi\Big(\frac{\sqrt{|z|}}{R}-t'\Big)\Big) \frac{\i}{10000 |c'|}\Big(\kappa_0 + \frac{\nu}{2\pi} \log \Big|\frac{z}{c'}\Big| \Big) |z|^{-1} dz\wedge d\bar z
\end{equation}
is still positive, but
\begin{equation}
\int_{X} \Big(\omega_{t, t'}^2-\frac{c^2}{2}\Omega\wedge\bar\Omega\Big) \to -\infty
\end{equation}
as $t'\to \infty$.
Since this integral is positive for $t'$ sufficiently small, we see that there exists a value $t'$ such that the integral becomes 0. For this value of $t'$, there exists a smooth bounded function $\varphi_5 : X \rightarrow \RR$  such that 
\begin{equation}
\omega_{\ALG^*} \equiv \omega_{t, t'}+\i\partial\bar\partial\varphi_5
\end{equation}
is hyperk\"ahler; see \cite[Theorem~1.1]{TianYau}. By Hein's estimates for higher order derivatives of $\varphi_5$ \cite[Proposition~2.9]{Hein}, $\omega_{\ALG^*}$ is asymptotic to $\omega_{t, t'}$, which is the same as $\omega_{\SF, \epsilon}[\sigma]$ near infinity. By Subsection \ref{ss:semi-flat},  the hyperk\"ahler structure is $\ALG^*$ according to 
Definition \ref{d:ALGstar}. Finally, the order can be improved to $2$ by \cite[Theorem~1.10]{CVZ2}.

\section{Topology and coordinates at infinity}
\label{s:Topology of gravitational instantons}
In this section, we will determine the topology of ALG and ALG$^*$ gravitational instantons, and also address the issue of existence of ``uniform'' coordinates at infinity. 

\subsection{Topology of gravitational instantons}

\label{ss:Topology of gravitational instantons}
In this subsection we will prove that any two gravitational instantons 
of type $\ALG$ or $\ALG^*$ with the same $\beta$ or $\nu$ are diffeomorphic, by a diffeomorphism with nice properties. Recall that a rational elliptic surface admits a section if and only if there are no multiple fibers; see for example \cite{HarbourneLang}.
\begin{theorem} 
\label{t:elldiff}
Let $z: S \rightarrow \PP^1$ and $z': S' \rightarrow \PP^1$ 
be rational elliptic surfaces without multiple fibers. Let $\Sigma, \Sigma'$ be sections of $z, z'$, 
and assume that $D, D'$ are singular fibers of the same Kodaira type ($\I_0^*$, $\II^*$, $\II$, $\III^*$, $\III$, $\IV^*$, $\IV$, $\I_1^*$, $\I_2^*$, $\I_3^*$, or $\I_4^*$) in $S, S'$, respectively, both over the point $p_{\infty} \equiv [0,1] \in \PP^1$. Then there exists a diffeomorphism 
\begin{align}
\Psi: S \setminus D \rightarrow S' \setminus D'
\end{align}
such that near infinity, $z = z' \circ \Psi$, the section $\Sigma$ is mapped to $\Sigma'$, and any point $w_1 \tau_1 + w_2 \tau_2$ on the fiber is mapped to $w_1 \tau'_1 +  w_2 \tau'_2$ on the fiber, where $\Sigma$ and $\Sigma'$ are viewed as the zero section $w_1=w_2=0$, $\tau_1, \tau_2$ are the periods of the fiber on $S$, and $\tau'_1, \tau'_2$ are the periods of the fiber on $S'$.
\label{t:Topology of gravitational instantons}
\end{theorem}
To prove this, we will use the Weierstra{\ss} representation of a rational elliptic surface, which we describe next.
Let $z: S \rightarrow \PP^1$ be any rational elliptic surface with a section $\Sigma \subset S$.
The section hits some multiplicity $1$ component of each singular fiber. The components in a singular fiber which do not hit this component (if there are any) form an ADE-configuration of $(-2)$-curves, which can be blown down to a surface $\check{S}$ with only RDP singularities. 
\begin{theorem}[Lecture II of \cite{MirandaBTES}] There exist a section $A$ of  $\mathcal{O}_{\PP^1}(4)$ and $B$ of $\mathcal{O}_{\PP^1}(6)$
such that the surface $\check{S}$ is biholomorphic to 
\begin{align}
S_{(A,B)} = \{ Z_2^2 Z_0 = Z_1^3 + A Z_1 Z_0^2 + B Z_0^3 \}  
\subset
\PP(  \mathcal{O}_{\PP^1} \oplus \mathcal{O}_{\PP^1}(2) \oplus \mathcal{O}_{\PP^1}(3))
\end{align}
by a biholomorphism which takes $z$ onto $\pi_{(A,B)} : S_{(A,B)} \rightarrow \PP^1$, which is the restriction to $S_{(A,B)}$ of the projection 
\begin{align} 
\pi : \PP(  \mathcal{O}_{\PP^1} \oplus \mathcal{O}_{\PP^1}(2) \oplus \mathcal{O}_{\PP^1}(3)) \rightarrow \PP^1.
\end{align}
The biholomorhism can be chosen to map the given section $\Sigma$ onto the zero section $\{Z_1 = 0, Z_0 = 0\}$.
Finally, this biholomorphism lifts to a biholomorphism of $S$ with the minimal resolution of $S_{(A,B)}$. 
\end{theorem}
The pair $(A,B)$ is called  \textit{Weierstra{\ss} data}.
The discriminant is
$\Delta = 4 A^3 + 27 B^2$.
 We let
\begin{align}
a_p &= \mbox{the order of vanishing of $A$ at $p$}, \\
b_p &= \mbox{the order of vanishing of $B$ at $p$}, \\
\delta_p & =  \mbox{the order of vanishing of $\Delta$ at $p$},\\
e_p &= \mbox{the Euler characteristic of the singular fiber}.
\end{align}

In \cite{Tate}, Tate showed that the vanishing orders $a_p, b_p, \delta_p$ completely determine the 
fiber type; see Table~\ref{Tatetable}. Note that we only consider Weierstra{\ss} data with the following constraint: there is no point $p \in \PP^1$ where $a_p \geq 4$ and $b_p \geq 6$ \cite[Proposition~III.3.2]{MirandaBTES}.

\begin{table}[h]
\caption{Tate's algorithm}
\label{Tatetable}
 \renewcommand\arraystretch{1.5}
\begin{tabular}{|c|c|c|c|c|c|c|c|c|c|c|c|} \hline
Type  & $\I_0$ & $\I_N$ & $\I_0^*$  & $\I_N^*$ &$\II$  &$\III$ &$\IV$ &$\IV^*$ &$\III^*$ & $\II^*$ \\  \hline
$e_p$  & $0$   & $N$ &  $6$  & $N+6$ & $2$  & $3$ & $4$  &$8$  &  $9$ &  $10$  \\  \hline
 $a_p$  &$a \geq 0$  & $0$ &  $a \geq 2 $  &  $2$ & $a \geq 1$    &  $1$&$a \geq 2$ & $a \geq 3$ &$3$ &  $a \geq 4$\\  \hline
 $b_p$ &  $b \geq 0$ &$0$ & $b \geq 3$   & $3$ &  $1$ & $b \geq 2$ &$2$ & $4$ & $b \geq 5$ & $5$  \\  \hline
$\delta_p$ &  $0$ & $N$ &   $6$ &$N+6$ &  $2$  & $3$ &  $4$ & $8$  &  $9$&  $10$  \\  \hline
  \end{tabular}
\end{table}
We can assume that the divisor $D = D_{\infty}$ is over the point $p_{\infty} \equiv [0,1] \in \PP^1$. We will next write out the constraints on $A,B$ in each case and prove that the space of Weierstra{\ss} data with a fixed Kodaira fiber type $D_{\infty}$ is connected in the following cases which arise in the compactification of $\ALG$ and $\ALG^*$ gravitational instantons. 

{\bf{Cases of I$_\nu^*$:}} In these cases,
\begin{align}
A &= a_0 z_1^4 + a_1 z_1^3 z_2 + a_2 z_1^2 z_2^2, \\
B & = b_0 z_1^6 + b_1 z_1^5 z_2 + b_2 z_1^4 z_2^2 + b_3 z_1^3 z_2^3,
\end{align}
and
\begin{align}
\begin{split}
\Delta = &( 4 a_2^3 + 27 b_3^2 ) z_1^6 z_2^6 + (12 a_1 a_2^2 + 54 b_2 b_3 ) z_1^7 z_2^5 \\
& + (12 a_1^2 a_2 + 12 a_0 a_2^2 + 54 b_1 b_3 + 27 b_2^2) z_1^8 z_2^4 \\
& + (4 a_1^3 + 24 a_0 a_1 a_2 + 54 b_0 b_3 + 54 b_1 b_2) z_1^9 z_2^3 \\
& + (12 a_0^2 a_2 + 12 a_0 a_1^2 + 27 b_1^2 +54 b_0 b_2) z_1^{10} z_2^2 \\
& + (12 a_0^2 a_1 + 54 b_0 b_1)z_1^{11} z_2 + (4 a_0^3 + 27 b_0^2) z_1^{12}.
\end{split}
\end{align}
In the I$_0^*$ case, we need to show that the set
\begin{align}\{4 a_2^3 + 27 b_3^2 \neq 0\}\subset \{(a_0, a_1, a_2, b_0, b_1, b_2, b_3) \in \CC^7\}\end{align}
is connected. This is trivially true because the inequalities do not affect the connectedness.

In the I$_1^*$ case, the set is
\begin{align}\{a_2 \neq 0, b_3 \neq 0, 4 a_2^3 + 27 b_3^2 = 0, 12 a_1 a_2^2 + 54 b_2 b_3 \neq 0\}.\end{align}
It is connected because the condition $4 a_2^3 + 27 b_3^2 = 0$ defines an elliptic curve in $\CC^2$, which is connected. The inequalities also do not affect the connectedness.

In the I$_2^*$ case, the set is
\begin{align}\{a_2 \neq 0, b_3 \neq 0, 4 a_2^3 + 27 b_3^2 = 12 a_1 a_2^2 + 54 b_2 b_3 = 0, 12 a_1^2 a_2 + 12 a_0 a_2^2 + 54 b_1 b_3 + 27 b_2^2 \neq 0\}.\end{align}
The condition
$12 a_1 a_2^2 + 54 b_2 b_3 = 0$
 is the same as
\begin{align}
b_2 = - \frac{12 a_1 a_2^2}{54 b_3}.
\end{align}
If we ignore the inequalities, we get a graph over the product of the elliptic curve $4 a_2^3 + 27 b_3^2 = 0$ with $\{(a_0, a_1, b_0, b_1) \in \CC^4\}$, so this set is also connected.

In the I$_3^*$ case, the set is
\begin{align}
\begin{split}
\{a_2 \neq 0, b_3 \neq 0, 4 a_2^3 + 27 b_3^2 = 12 a_1 a_2^2 + 54 b_2 b_3 = 12 &a_1^2 a_2 + 12 a_0 a_2^2 + 54 b_1 b_3 + 27 b_2^2 = 0,  \\
&4 a_1^3 + 24 a_0 a_1 a_2 + 54 b_0 b_3 + 54 b_1 b_2 \neq 0\}.
\end{split}
\end{align}
The condition
\begin{align}
12 a_1^2 a_2 + 12 a_0 a_2^2 + 54 b_1 b_3 + 27 b_2^2 = 0
\end{align}
is the same as
\begin{align}
 b_1 = - \frac{12 a_1^2 a_2 + 12 a_0 a_2^2 + 27 b_2^2}{54 b_3}.
\end{align}
If we ignore the inequalities, we get a graph of a map from the product of the elliptic curve $4 a_2^3 + 27 b_3^2 = 0$ with $\{(a_0, a_1, b_0) \in \CC^3\}$ to $\{(b_1, b_2) \in \CC^2\}$. To define this map, we compute $b_2$ first and then compute $b_1$, so this set is also connected.

In the I$_4^*$ case, the set is
\begin{align}
\begin{split}
\{a_2 \neq 0, &b_3 \neq 0, 4 a_2^3 + 27 b_3^2 = 12 a_1 a_2^2 + 54 b_2 b_3 = 12 a_1^2 a_2 + 12 a_0 a_2^2 + 54 b_1 b_3 + 27 b_2^2 = 0,  \\
&4 a_1^3 + 24 a_0 a_1 a_2 + 54 b_0 b_3 + 54 b_1 b_2 = 0 , 12 a_0^2 a_2 + 12 a_0 a_1^2 + 27 b_1^2 +54 b_0 b_2 \neq 0\}.
\end{split}
\end{align}
The condition
\begin{align}
4 a_1^3 + 24 a_0 a_1 a_2 + 54 b_0 b_3 + 54 b_1 b_2 = 0
\end{align}
is the same as
\begin{align}
b_0 = - \frac{4 a_1^3 + 24 a_0 a_1 a_2 + 54 b_1 b_2}{54 b_3}.
\end{align}
If we ignore the inequalities, we get a graph of a map from the product of the elliptic curve $4 a_2^3 + 27 b_3^2 = 0$ with $\{(a_0, a_1) \in \CC^2\}$ to $\{(b_0, b_1, b_2) \in \CC^3\}$. To define this map, we compute $b_2$ first, $b_1$ second, and then compute $b_0$, so this set is also connected.

\vspace{2mm}
\noindent
{\bf{Case of II:}}
In this case,
\begin{align}
A &= a_0 z_1^4 + a_1 z_1^3 z_2 + a_2 z_1^2 z_2^2 + a_3 z_1 z_2^3, \\
B & = b_0 z_1^6 + b_1 z_1^5 z_2 + b_2 z_1^4 z_2^2 + b_3 z_1^3 z_2^3 + b_4 z_1^2 z_2^4 + b_5 z_1 z_2^5, \\
\Delta & = 27 b_5^2 z_1^2 z_2^{10} + \dots
\end{align}
The set
\begin{align}
\{b_5 \neq 0, 27 b_5^2 \neq 0\}\subset\{(a_0, ..., a_3, b_0, ..., b_5) \in \CC^{10}\}
\end{align}
is obviously connected.

\vspace{2mm}
\noindent
{\bf{Case of III:}}
In this case,
\begin{align}
A &= a_0 z_1^4 + a_1 z_1^3 z_2 + a_2 z_1^2 z_2^2 + a_3 z_1 z_2^3, \\
B & = b_0 z_1^6 + b_1 z_1^5 z_2 + b_2 z_1^4 z_2^2 + b_3 z_1^3 z_2^3 + b_4 z_1^2 z_2^4, \\
\Delta & = 4 a_3^3 z_1^3 z_2^9 + \dots
\end{align}
The set
\begin{align}
\{a_3 \neq 0, 4 a_3^3 \neq 0\}\subset\{(a_0, ..., a_3, b_0, ..., b_4) \in \CC^{9}\}
\end{align}
is obviously connected.

\vspace{2mm}
\noindent
{\bf{Case of IV:}}
In this case,
\begin{align}
A &= a_0 z_1^4 + a_1 z_1^3 z_2 + a_2 z_1^2 z_2^2, \\
B & = b_0 z_1^6 + b_1 z_1^5 z_2 + b_2 z_1^4 z_2^2 + b_3 z_1^3 z_2^3 + b_4 z_1^2 z_2^4, \\
\Delta & = 27 b_4^2 z_1^4 z_2^8 + \dots
\end{align}
The set
\begin{align}
\{b_4 \neq 0, 27 b_4^2 \neq 0\}\subset\{(a_0, ..., a_2, b_0, ..., b_4) \in \CC^{8}\}
\end{align}
is obviously connected.

\vspace{2mm}
\noindent
{\bf{Case of IV$^*$:}}
In this case,
\begin{align}
A &= a_0 z_1^4 + a_1 z_1^3 z_2, \quad
B  = b_0 z_1^6 + b_1 z_1^5 z_2 + b_2 z_1^4 z_2^2, \quad
\Delta  = 27 b_2^2 z_1^8 z_2^4 + \dots
\end{align}
The set
\begin{align}
\{b_2 \neq 0, 27 b_2^2 \neq 0\}\subset\{(a_0, a_1, b_0, b_1, b_2) \in \CC^{5}\}
\end{align}
is obviously connected.

\vspace{2mm}
\noindent
{\bf{Case of III$^*$:}}
In this case,
\begin{align}
A &= a_0 z_1^4 + a_1 z_1^3 z_2, \quad
B  = b_0 z_1^6 + b_1 z_1^5 z_2, \quad
\Delta  = 4 a_1^3 z_1^9 z_2^3 + \dots
\end{align}
The set
\begin{align}
\{a_1 \neq 0, 4 a_1^3 \neq 0\}\subset\{(a_0, a_1, b_0, b_1) \in \CC^{4}\}
\end{align}
is obviously connected.

\vspace{2mm}
\noindent
{\bf{Case of II$^*$:}}
In this case,
\begin{align}
A &= a_0 z_1^4, \quad B  = b_0 z_1^6 + b_1 z_1^5 z_2, \quad
\Delta  = 27 b_1^2 z_1^{10} z_2^2 + \dots
\end{align}
The set
\begin{align}
\{b_1 \neq 0, 27 b_1^2 \neq 0\}\subset\{(a_0, b_0, b_1) \in \CC^{3}\}
\end{align}
is obviously connected.

\begin{definition} The Weierstra{\ss} data satisfying the above conditions in each case will be called {\textit{allowable}}. Allowable Weierstra{\ss} data $(A,B)$ is called {\textit{good}} 
if $\Delta$ has no multiple root on $\PP^1 \setminus \{ [0,1] \}$. 
Otherwise, it is called {\textit{bad}}.
\end{definition}

\begin{proof}[Proof of Theorem~\ref{t:elldiff}]
We first prove the theorem assuming that $z : S \rightarrow \PP^1$
and $z' : S' \rightarrow \PP^1$ both correspond to good Weierstra{\ss} data. 
Note that good Weierstra{\ss} data is equivalent 
to all singular fibers being of type I$_1$ (except for $D_{\infty}$).
Thus the Weierstra{\ss} models are smooth away from a single RDP singularity over $p_{\infty}$. 
The good property is characterized by non-vanishing of the resultant $R(\Delta, \Delta')$ on $\PP^1 \setminus \{[0,1]\}$. This resultant is not identically zero on the space of allowable Weierstra{\ss} data, since otherwise it would mean that there is no rational elliptic surface with fiber $D_{\infty}$ and all other fibers of type I$_1$. This is a contradiction since such surfaces definitely exist; see \cite{Miranda1990}. Consequently, the bad condition defines a proper subvariety of the space of 
allowable Weierstra{\ss} data, so does not affect the connectedness.  
Consequently, there is a holomorphic family corresponding to good Weierstra{\ss} data 
\begin{align} 
S_{(A_t,B_t)} \rightarrow \mathcal{S} \rightarrow \mathcal{B}
\end{align}
over a Zariski-open base space $\mathcal{B}$. 
To prove the theorem, clearly we can reduce to the case that the base $\mathcal{B} = B_2(0)\ \subset \CC$ is a disc, $S_{(A,B)} = \pi^{-1}(0)$, and $S_{(A',B')} = \pi^{-1}(1)$. 
Since the $D_{\infty}$ fibers of the Weierstra{\ss} models are all the same cuspidal cubic, by taking a fixed sequence of blow-ups of the ambient space 
$\mathbb{P}( \mathcal{O}_{\PP_1} + \mathcal{O}_{\PP^1}(2) + \mathcal{O}_{\PP^1}(3))$ at the 
singular point of the cuspidal cubic, and then taking the corresponding sequences of proper transforms of $S_{(A_t,B_t)}$, we can resolve this family to a family with smooth total space $\tilde{\mathcal{S}}$, with compact fiber over $t$ given by $\tilde{S}_{(A_t,B_t)}$, 
which is the minimal resolution of $S_{(A_t,B_t)}$. Let $z_t : \tilde{S}_{(A_t,B_t)} \rightarrow \PP^1$ denote the elliptic projection.
By choosing some Riemannian metric on $\tilde{\mathcal{S}}$, we can therefore find a vector field $Y_1$ on $\tilde{\mathcal{S}}$, defined on a subset of $\tilde{\mathcal{S}}$ containing a neighborhood of $\tilde{D}_{\infty}$ of each fiber of the family, whose vertical part vanishes on $\tilde{D}_{\infty}$ and on each section $\Sigma_0$, and whose horizontal part projects to $\p/\p t$,
and which satisfies $(z_t)_*(Y_1) = 0$ for each $t \in \mathcal{B}$. We can also find a vector field $Y_2$ on $\tilde{\mathcal{S}}$, defined on a subset of $\tilde{\mathcal{S}}$ containing a large compact subset of each fiber of the family  whose horizontal part projects to $\p/\p t$.  Choosing a cutoff function $\chi$ which is $0$ near infinity and $1$ on the compact region of each fiber, the vector field $Y = \chi Y_2 + (1- \chi) Y_1$ will then give a globally defined vector field on $\tilde{\mathcal{S}}$. 
Similar to \cite[Theorem~4.1]{MorrowKodaira}, the flow of this vector field will then provide a smooth family of diffeomorphisms 
$F_t : \tilde{S}_{(A_t,B_t)} \rightarrow \tilde{S}_{(A_0,B_0)}$
satisfying $z_t \circ F_t = z_0$ and preserving the zero section near infinity. 

We next modify $F_t$ so that it has the affine property near infinity. 
Fixing a point $p \in \PP^1$ near infinity, we can use $F_t$ to identify 
$ H^1( z_t^{-1} (p);\ZZ)$ with $H^1( z_0^{-1} (p);\ZZ)$, and we can make 
a smooth choice of basis
\begin{align}
[\gamma_1(p,t)], [\gamma_2(p,t)] \in H^1( z_t^{-1} (p);\ZZ).
\end{align}
Next, the diffeomorphisms $F_t$ give an identification of the homological invariants
\begin{align}
R^1 (z_t)_* \ZZ  \cong R^1( z_0)_* \ZZ
\end{align}
near infinity. Given any other point $q$ near infinity, we choose paths $h(z,t)$ from $p$ to $q$, and then use the homological invariants to identify 
$H^1( z_t^{-1} (p);\ZZ)$ with $H^1( z_t^{-1} (q);\ZZ)$ to define
\begin{align}
[\gamma_1(q,t)], [\gamma_2 (q,t)] \in H^1(  z_t^{-1} (q);\ZZ).
\end{align}
These depend upon the path, but are well-defined up to monodromy.  
After using $F_t$ to identify $H^1(  z_t^{-1} (q);\ZZ)$ with $H^1(  z_0^{-1} (q);\ZZ)$, we see that this choice varies smoothly in $t$. 

Next, letting $x = Z_1/Z_0$, $y = Z_2/Z_0$, we see that $\alpha_1 = dx/y$ is a smooth holomorphic 1-form on the total space of the Weierstra{\ss} moduli space (for the same reason that $dx/y$ is a globally defined holomorphic $1$-form on the elliptic curve $y^2 = x^2 + a x  + b$). 
We then define multi-valued period functions by
 \begin{align}
\tau_1(q,t) = \int_{\gamma_1(q,t)} \alpha_1, \quad \tau_2(q,t) = \int_{\gamma_2(q,t)} \alpha_1,
\end{align}
which are smooth in both $q$ and $t$.

Since each $z_t : \tilde{S}_{(A_t,B_t)} \rightarrow \PP^1$  is an elliptic surface, using \cite{Kodaira1963}, for each $t$, there exists a neighborhood $U_{\infty,t}$ 
of $\tilde{D}_{\infty}$ such that we can identify 
$U_{\infty,t} \setminus \tilde{D}_{\infty}$ with a standard model
\begin{align}
\{\xi \in \CC \ | \ 0 < |\xi| < R^{-1} \}\times \CC/(\ZZ\oplus \ZZ),
\end{align}
where $\ZZ\oplus \ZZ$ acts by $(m,n) \cdot (\xi, w)=(\xi, w + m\tau_1(\xi^{-1}, t)+n\tau_2(\xi^{-1}, t))$,
such that $\xi=z^{-1}$ and the given sections $\Sigma$, $\Sigma'$ are mapped to $\Sigma_0 = \{(\xi,0) \ |0 < |\xi| < R^{-1} \}$ in these coordinates. 
We can define a diffeomorphism $\Phi_t$ near infinity from  $S_t \setminus D_t$
to $S_0 \setminus D_0$  such that $z_t = z_0 \circ \Phi_t$, the section $\Sigma_t$ is mapped to $\Sigma_0$, and any point $w_1 \tau_1(q,t) + w_2 \tau_{2}(q,t)$ on the fiber is mapped to $w_1 \tau_{1}(q,0) +  w_2 \tau_{2}(q,0)$ on each fiber.
Then the differential of $\Phi_t$ determines a vector field $Y_3$ on $\tilde{\mathcal{S}}$, defined on a subset of $\tilde{\mathcal{S}}$ containing a neighborhood of $\tilde{D}_{\infty}$ of each fiber of the family, whose horizontal part projects to $\p/\p t$, and which is affine near infinity. If we use $Y_3$ instead of $Y_1$ in the definition of $F_t$, then we get the required diffeomorphism which is affine near infinity.

Next, we consider the case that $z$ arises from bad Weierstra{\ss} data $(A,B)$. 
Above we have shown that the space of allowable Weierstra{\ss} data in each 
case is connected, and the bad Weierstra{\ss} data is a proper subvariety of the good Weierstra{\ss} data. We can therefore find a holomorphic 
family 
\begin{align} 
{S}_{(A_t,B_t)} \rightarrow \mathcal{S} \rightarrow \mathcal{B}
\end{align}
over the base $\mathcal{B}$ a disc, with $(A_0,B_0) = (A,B)$ and
 $(A_t,B_t)$ good for $t \neq 0$.
Arguing as in the above case, since the $D_{\infty}$ fibers of the Weierstra{\ss} models are all the same cuspidal cubic, we can resolve the family $\mathcal{S}$ to a family $\tilde{\mathcal{S}}$ with smooth fibers, except for the central fiber, which will have some nontrivial RDP singularities away from $\tilde{D}_{\infty} \subset \tilde{S}_{(A_0,B_0)}$, since we assumed the Weierstra{\ss} data $(A,B)$ is bad.
Consequently, the family of surfaces $\tilde{S}$ is then a \textit{smoothing} of all the RDP 
singularities of $S_{(A_0,B_0)}$. From the existence of simultaneous resolutions for RDP singularities \cite{Brieskorn,Tjurina}, after a base change, we may lift this family to a smooth family, with central fiber the minimal resolution of $S_{(A_0,B_0)}$.
The proof then proceeds as in the previous case to show that we can find a diffeomorphism 
from any rational elliptic surface with bad Weierstra{\ss} data to a rational elliptic surface with good Weierstra{\ss} data which is standard near infinity. 

Finally, we can obtain a diffeomorphism between any two surfaces with bad Weierstra{\ss} data by first finding a diffeomorphism from each surface to a surface with good Weierstra{\ss} data as in the previous paragraph, and then finding a diffeomorphism between these good surfaces as in the first part of the above proof.
\end{proof}

\begin{remark}In the case that $D$ is a regular fiber, Theorem~\ref{t:elldiff} is equivalent to \cite[Theorem~3.4]{CCIII}. Theorem~\ref{t:elldiff} implies the following.  
In the case that $D$ is an $\I_0^*$, $\II^*$, $\II$, $\III^*$, $\III$, $\IV^*$, or $\IV$ fiber, we have that any two ALG gravitational instantons with the same tangent cone at infinity are diffeomorphic to each other, which is Theorem~\ref{t:ALGtopology}. In the case that $D$ is an $\I_\nu^*$ fiber for $\nu=1, 2, 3, 4$, we have that any two ALG$^*$ gravitational instantons with the same $\nu$ are diffeomorphic to each other, which is Theorem~\ref{t:ALGstartopology}. In the following subsection, we will make an improvement to this which takes into account the $\ALG$ or $\ALG^*$ structure near infinity. 
\end{remark}

\subsection{Coordinates at infinity}
In this subsection, we prove Theorem~\ref{t:ALGstaruniform} and Theorem~\ref{t:ALGuniform} from the introduction, regarding the existence of ``uniform'' coordinate systems for ALG$^*$ and ALG gravitational instantons.  
\begin{proof}[Proof of Theorem \ref{t:ALGuniform}]
In the following arguments, the constant $R$ may increase in each step, but for simplicity of notation we will just use a single constant $R$.
Let 
\begin{equation}
(z_{\cC}, w_{\cC}) \equiv  (\mathscr{U}^{\frac{1}{\beta}}, \mathscr{U}^{1 - \frac{1}{\beta}} \mathscr{V})
\end{equation}
be the coordinates on $\cC_{\beta,\tau,L}(R)$ and $z$ be the holomorphic function on $X$ asymptotic to $z_{\cC}$ which was proved to exist in \cite[Theorem~4.14]{CCII}. Then for large $R$, we can define a smooth map 
\begin{equation}
f_1: \CC \setminus \overline{B_{(2R)^{\frac{1}{\beta}}}(0)} \to \CC \setminus \overline{B_{R^{\frac{1}{\beta}}}(0)}
\end{equation}
by the following. For any number $t \in \CC$, we define $p$ as the point with $z_{\cC} = t$ and $w_{\cC}=0$. Then $f_1(t, \bar t)$ is defined as $z(\Phi(p))$. 
Since $g$ is ALG of order $\mathfrak{n}$ in the $\Phi$ coordinates, it follows that
\begin{equation}
|f_1(t, \bar t) - t| = O \big(|t|^{1+\beta(-\mathfrak{n}+\epsilon)} \big)
\end{equation}
and
\begin{equation}
\label{pf1}
|\partial_t f_1(t, \bar t) - 1| + |\partial_{\bar t} f_1(t, \bar t)| = O \big(|t|^{\beta(-\mathfrak{n}+\epsilon)}\big)
\end{equation}
for any $\epsilon>0$.

Now we choose a function $\chi_1$ which is $1$  if $|t|^{\beta} \ge 20 R$ and is $0$ if $|t|^{\beta} \le 10 R$. Define 
\begin{equation}
f_2(t, \bar t) = \chi_1(t, \bar t) (f_1(t, \bar t) - t) + t.
\end{equation}
Then $|\partial_t f_2(t, \bar t)|$ is very close to $1$ and $|\partial_{\bar t} f_2(t, \bar t)|$ is very close to $0$, so $f_2$ is a diffeomorphism which is $f_1$ if $|t|^{\beta} \ge 20 R$  and is the identity map if $|t|^{\beta} \le 10 R$. Then we define a map 
\begin{equation}
F_1: \cC_{\beta, \tau, L}(3R) \to X \setminus X_R
\end{equation}
by the following way: for each point $(z_{\cC}, w_{\cC}) \in \cC_{\beta, \tau, L}(3R)$, we write the fiber $\{z = f_2(z_{\cC})\}\subset X \setminus X_R$ as $\CC/ (\ZZ \tau_1 \oplus \ZZ \tau_2)$ such that the zero point is $\Phi(f_1^{-1}(f_2(z_{\cC})), 0)$. Then
\begin{equation}
F_1(z_{\cC}, w_{\cC}) = w_1 \tau_1 + w_2 \tau_2
\end{equation}
on $\{z=f_2(z_{\cC})\}$ if 
\begin{equation}
w_{\cC} = w_1 \tau_{1,\model} + w_2 \tau_{2,\model},
\end{equation}
where 
\begin{equation}
\tau_{1,\model} = \mathscr{U}^{1 - \frac{1}{\beta}} \cdot L
\end{equation}
and 
\begin{equation}
\tau_{2,\model} = \mathscr{U}^{1 - \frac{1}{\beta}} \cdot L \cdot \tau.
\end{equation}

Since $g$ is ALG of order $\mathfrak{n}$ in the $\Phi$ coordinates, we also have the estimates
\begin{align}
|\tau_1 - \tau_{1,\model}| \le C \cdot |\tau_{1,\model}| \cdot |\mathscr{U}|^{-\mathfrak{n} + \epsilon}, \quad |\tau_2 - \tau_{2,\model}| = C \cdot |\tau_{2,\model}| \cdot |\mathscr{U}|^{-\mathfrak{n} + \epsilon},
\end{align}
for any $\epsilon >0$, with similar estimates on their derivatives. Therefore we can choose multi-valued functions $\tau_1$ and $\tau_2$ so that $F_1$ is a well-defined single-valued smooth mapping. The point $F_1(p)$ is very close to $\Phi(p)$ if $|\mathscr{U}(p)| \ge 20 R$, since the distance between $F_1(p)$ and $\Phi(p)$ is bounded by a constant multiple of the diameter of the fiber times $|\mathscr{U}|^{- \mathfrak{n} + \epsilon}$ for any $\epsilon>0$. Since $F_1(p)$ is within the injectivity radius at $\Phi(p)$, we let $Y(p)$ be the tangent vector at $\Phi(p)$ corresponding to the shortest geodesic connecting these points. Then $F_1(p)=\exp_{\Phi(p)}(Y(p))$. Note that if we parallel transport through this shortest geodesic, then the differential of $\Phi$ is very close to the diffential of $F_1$. Choose another cut-off function $\chi_2$ which is $0$ if $|\mathscr{U}| \ge 200 R$ and is $1$ if $|\mathscr{U}| \le 100 R$. Then the map 
\begin{equation}
F_2(p) \equiv \exp_{\Phi(p)} (\chi_2(p) \cdot Y(p))
\end{equation}
is a diffeomorphism which is $\Phi$ if $|\mathscr{U}| \ge 200 R$ and is $F_1$ if $3R \le |\mathscr{U}| \le 100 R$.
We can define the map $F_2' : \cC_{\beta,\tau,L}(3R) \to X' \setminus X'_R$ in a similar way. 

By Theorem \ref{t:Topology of gravitational instantons}, there exists a diffeomorphism $\Psi$ from $X$ to $X'$ such that when $|z|^\beta>R$, $z = z' \circ \Psi$. Moreover, $\Psi$ preserves the affine structure on each fiber near infinity. Then we need to study the map $F_3 \equiv (F_2')^{-1} \circ \Psi \circ F_2$. The zero section over the circle $\{|z|^\beta = 4R\}$ is mapped by $F_3$ to another smooth section over the circle $\{|z|^\beta = 4R\}$. It is a section of the $T^2$-fibration. If it is liftable to a smooth section of the line bundle $\mathbb{L}^{-1}$ (the line bundle whose fiber at each point is the universal cover of the $T^2$-fiber), then $F_3$ also maps the zero section over $\{5R \ge |z|^\beta \ge 4R\}$ to the projection of a section $s$ of  $\mathbb{L}^{-1}$. Choose a cut-off function $\chi_3$ which is 0 when $|z|^\beta \ge 5R$ and is 1 when $|z|^\beta \le 4R$. Then for 
\begin{equation}
F_4(z, w_1\tau_1 + w_2\tau_2) \equiv (z, \chi_3 \cdot s(z,\bar z) + w_1\tau_1 + w_2\tau_2),
\end{equation}
the map 
\begin{align}
F_5(p) \equiv
\begin{cases}
F_2' \circ F_4 \circ (F_2)^{-1}(p)  & |z(p)| > R\\
\Psi(p) & |z(p)| \leq R \mbox{ or } p \in X_R\\
\end{cases}
\end{align}
will be $\Psi$ in the interior and $\Phi' \circ (\Phi)^{-1}$ near infinity, which is the the required mapping.

However, in general, this section may not be liftable to a section of $\mathbb{L}^{-1}$. To solve this problem, if we write $\mathscr{U}=z^{\beta}$ as $r e^{\i\theta}$, then we can rewrite the torus fibration over $\{|\mathscr{U}|\ge R\}$ as the trivial product of $(r, \theta) \in [R,\infty)\times[0, 2 \pi \beta]$ with $T^2$ after gluing the $T^2$ for $\theta = 0$ with the $T^2$ for $\theta = 2 \pi \beta$ using the monodromy group. Therefore, the periods $\tau_1$ and $\tau_2$ on $\theta = 0$ can be identified with $a_{11}\tau_1+a_{12}\tau_2$ and $a_{21}\tau_1+a_{22}\tau_2$ using the monodromy group. The difference of two local sections on this $T^2$-fibration can be viewed as a section of an $\RR^2$-fibration
\begin{equation}
\sigma = x_1(r,\theta) \tau_1 + x_2(r,\theta) \tau_2
\end{equation} on $[R,\infty)\times[0, 2 \pi \beta]$, with the gluing condition that 
there exist integers $m, n$ such that
\begin{equation}
x_1(r,0) (a_{11}\tau_1+a_{12}\tau_2) + x_2(r,0)(a_{21}\tau_1+a_{22}\tau_2) - x_1(r, 2 \pi \beta) \tau_1 - x_2(r,2 \pi \beta) \tau_2 = m \tau_1 + n \tau_2.
\end{equation}
In general, $(m,n)$ may not be $(0,0)$. Fortunately, since the matrix
\begin{align}
\left(
\begin{matrix}
a_{11}-1 & a_{12} \\
a_{21} & a_{22}-1 \\
\end{matrix}
\right)
\end{align}
is invertible in every ALG case, there exists constants $C_1$, $C_2$ such that
\begin{equation}
C_1 (a_{11}\tau_1+a_{12}\tau_2) + C_2(a_{21}\tau_1+a_{22}\tau_2) - C_1 \tau_1 - C_2 \tau_2 = m \tau_1 + n \tau_2.
\end{equation}
Therefore, the map $F : \cC_{\beta, \tau, L}(R) \to \cC_{\beta, \tau, L}(R)$ defined by 
\begin{equation}
F: (z, w_1 \tau_1(z) + w_2 \tau_2(z) ) 
\mapsto (z, (w_1 + C_1) \tau_1(z) + (w_2 + C_2) \tau_2(z) )
\end{equation}
is well-defined and $F^{-1} \circ F_3$ now maps the zero section to the projection of a section of $\mathbb{L}^{-1}$. From the definition of Greene-Shapere-Vafa-Yau's semi-flat metric, the map $F$ preserves the hyperk\"ahler structure on the model space, and we are done. 
\end{proof}

\begin{proof}[Proof of Theorem \ref{t:ALGstaruniform}]
The proof is similar to the proof of Theorem \ref{t:ALGuniform} with the following minor modifications. In the first part, we use the holomorphic function $z$ studied in the proof of Proposition \ref{p:fircom}. We define $f_1$ similarly, and since $g$ is ALG$^*$ of order $2$ in the $\Phi$ coordinates, we obtain an estimate 
\begin{equation}
|f_1(t, \bar t) - t| = O \big(|t|^{\frac{\epsilon}{2}}\big)
\end{equation}
for any $\epsilon > 0$.
We also have an estimate analogous to \eqref{pf1}. Since $g$ is $\ALG^*$ of order $2$ in the $\Phi$ coordinates, we have the estimates
\begin{align}
|\tau_1 - \tau_{1,\model}| \le C \cdot |\tau_{1,\model}| \cdot r^{- 2 + \epsilon}, \quad |\tau_2 - \tau_{2,\model}| = C \cdot |\tau_{2,\model}| \cdot r^{- 2 + \epsilon},
\end{align}
for any $\epsilon > 0$, with similar estimates on their derivatives, 
where $\tau_{1,\model}$ and $\tau_{2,\model}$ were defined in \eqref{Istartau}, so we can define the mapping $F_1$ in the same way. 
The distance between $F_1(p)$ and $\Phi(p)$ is bounded by a constant multiple of the diameter of the fiber times $r^{- 2 + \epsilon}$ for any $\epsilon>0$.
Since  the diameter of the fiber is proportional to $\sqrt{\log r}$, 
and the injectivity radius is proportional to $1/\sqrt{\log r}$, the construction of the mapping $F_2$ is also similar. In the last part, we only used the fact that the monodromy generator minus the identity matrix is invertible, which from \eqref{Istarmono} obviously holds in this case.
\end{proof}

\section{Remarks on the order of ALG gravitational instantons}
\label{s:order}
In this section, we will prove Theorem \ref{t:order}. Let $(X_{\beta},g,\bm{\omega})$
be an $\ALG$ gravitational instanton of order $\mathfrak{n} > 0$, and let  $s: X \to [1 ,\infty)$ be a smooth extension of $r$ via the diffeomorphism $\Phi: \cC_{\beta,\tau,L}(R) \rightarrow X \setminus X_R$, where $\Phi$ and $r$ are defined as in Definition~\ref{d:ALG-space}. 
\begin{definition}
Let $(X, g, \bm{\omega})$ be an $\ALG$ gravitational instanton of order $\mathfrak{n} > 0$.	
For any fixed $\delta\in\dR$, the weight function $\hat{\varrho}_{\delta}$ on $X$ is defined by
$\hat{\varrho}_{\delta}\equiv s^{-\delta - 1}$, and we define the weighted Sobolev norms as follows:
\begin{align*}
 \|\omega\|_{L^2_{\delta}(X)} \equiv \Big(\int_{X} |\omega\cdot \hat{\varrho}_{\delta}|^2 \dvol_{X}\Big)^{\frac{1}{2}},\  
 \|\omega\|_{W^{k,2}_{\delta}(X)} \equiv  \Big(\sum_{m=0}^{k} \|\nabla^m\omega\|_{L_{\delta-m}^2(X)}^2\Big)^{\frac{1}{2}}.\end{align*}
\end{definition}
We remark that this convention differs from the convention in \cite[Definition~4.1]{CCI}, because of the following. 
\begin{lemma}
\label{l:Sobolev} Let $(X,g,\bm{\omega})$ be an $\ALG$  gravitational instanton of order $\mathfrak{n} > 0$. For any $\delta \in \RR$ and $k\in\dN_0$, there exists a constant $C_{k, \delta} > 0$ so that
 \begin{align}
\sum\limits_{m = 0}^k  \sup_{\bm{x} \in X} | (s(\bm{x}))^{m-\delta} \nabla^m \omega (\bm{x})|
\leq C  \Vert \omega \Vert_{ W^{k + 3,2}_{\delta}(X)} 
 \end{align}
for all $\omega \in  W^{k + 3,2}_{\delta}(X)$.\end{lemma}
\begin{proof}The proof is a standard rescaling argument; see for example~\cite[Proposition~6.16]{CVZ}.
\end{proof}

In \cite[Theorem~A]{CCII}, Chen-Chen proved that there exists a diffeomorphism $F_1: \cC_{\beta,\tau,L}(R) \rightarrow \cC_{\beta,\tau,L}(R)$ homotopic to the identity such that when the cone angle $\beta \leq 1/2$, the $\ALG$ order $\mathfrak{n}$ must be at least 2 in the $\Phi\circ F_1$ coordinates. On the other hand, when $1/2 < \beta < 1$, the leading term of $(\Phi\circ F_1)^*\bm{\omega}-\bm{\omega}^{\cC}$ must be a linear combination of closed anti-self-dual forms, which are $\Rea \mathscr{U}^{\frac{1}{\beta}-2}d \mathscr{U}\wedge d\mathscr{\bar V}$
and $\Ima \mathscr{U}^{\frac{1}{\beta}-2}d \mathscr{U}\wedge d\mathscr{\bar V}$.
Using Remark~\ref{r:IJK}, the following formulas are straightforward to verify:
\begin{align}\mathcal{L}_{Y}\omega_1^{\cC}&= - d(I_{\cC}^*\eta) = d(\Ima \mathscr{U}^{\frac{1}{\beta}-1}d \mathscr{U})=0, \\
\mathcal{L}_{Y}\omega_2^{\cC}&= - d(J_{\cC}^*\eta) = d(\Rea \mathscr{U}^{\frac{1}{\beta}-1}d \mathscr{\bar V})= \Big(\frac{1}{\beta}-1\Big)\Rea (\mathscr{U}^{\frac{1}{\beta}- 2}d \mathscr{U}\wedge d \mathscr{\bar V}),\\
\mathcal{L}_{Y}\omega_3^{\cC} &= - d (K_{\cC}^*\eta) = - d(\Ima \mathscr{U}^{\frac{1}{\beta}-1}d \mathscr{\bar V}) = - \Big(\frac{1}{\beta}-1\Big)\Ima (\mathscr{U}^{\frac{1}{\beta}- 2}d \mathscr{U}\wedge d \mathscr{\bar V}),
\end{align}
where  $\mathcal{L}$ denotes the Lie derivative and $Y$ is the $g^{\cC}$-metric dual to 
$\eta \equiv \Rea \mathscr{U}^{\frac{1}{\beta}-1}d \mathscr{U}$.
Furthermore, we have
\begin{align}
\mathcal{L}_{I_{\cC} Y}\omega_1^{\cC} &= - d \eta=0,  &\mathcal{L}_{I_{\cC} Y}\omega_2^{\cC} &= d(K_{\cC}^* \eta), & \mathcal{L}_{I_{\cC} Y}\omega_3^{\cC} &= - d (J_{\cC}^*\eta),\\
\mathcal{L}_{J_{\cC} Y}\omega_1^{\cC} &= - d (K_{\cC}^*\eta),  &\mathcal{L}_{J_{\cC} Y}\omega_2^{\cC} &= - d \eta = 0, & \mathcal{L}_{J_{\cC} Y}\omega_3^{\cC} &= d (I_{\cC}^*\eta) = 0,\\
\mathcal{L}_{K_{\cC} Y}\omega_1^{\cC} &= d (J_{\cC}^*\eta),  &\mathcal{L}_{K_{\cC} Y}\omega_2^{\cC} &= - d (I_{\cC}^* \eta) = 0, & \mathcal{L}_{K_{\cC} Y}\omega_3^{\cC} &= - d \eta = 0.\end{align}
These formulas imply that there exists another diffeomorphism $F_2: \cC_{\beta,\tau,L}(R) \rightarrow \cC_{\beta,\tau,L}(R)$ homotopic to the identity such that
\begin{align}
\label{o1g}
(\Phi\circ F_2)^*\omega_1&=\omega_{1}^{\cC} + O (r^{\frac{1}{\beta}- 2 - \epsilon}), \\
\label{o2g}
(\Phi\circ F_2)^*\omega_2&=\omega_{2}^{\cC} + O (r^{\frac{1}{\beta}- 2 - \epsilon}),
\end{align}
as $r \to \infty$, for some $\epsilon > 0$.  
To see this, for a vector field $Z = O(r^{\frac{1}{\beta}-1})$, let $\chi$ be a cut-off function which is 1 on $\cC_{\beta,\tau,L}(2R)$ and is supported on $\cC_{\beta,\tau,L}(R)$, and then define 
$\Phi_{Z,t} : \cC_{\beta,\tau,L}(R) \rightarrow \cC_{\beta,\tau,L}(R)$ by 
$\Phi_{Z,t}(p) = \exp_{g^{\cC},p}(t \cdot \chi \cdot Z_p)$. If $0 \leq t \leq 1$, since $\frac{1}{\beta}-1 < 1$, the vector $t \chi Z_p$ has norm much smaller than the conjugate radius at $p$ (which is comparable to $r(p)$), so $\Phi_{Z,t}$ is a diffeomorphism for $R$ sufficiently large. Then it is straightforward to see that we have the expansions
\begin{align}
\Phi_{J_{\cC}Y,1}^* {\omega}_1^{\cC} - {\omega}_1^{\cC} - \mathcal{L}_{J_{\cC}Y} \omega_1^{\cC}  & = O( r^{\frac{1}{\beta}-2 - \epsilon}), \\
\Phi_{K_{\cC}Y,1}^* {\omega}_1^{\cC} - {\omega}_1^{\cC} - \mathcal{L}_{K_{\cC}Y} \omega_1^{\cC}  &= O(r^{\frac{1}{\beta}-2 -\epsilon }),
\end{align}
as $r = |\mathscr{U}| \to \infty$, so we can arrange that \eqref{o1g} is satisfied after pulling back by an appropriate diffeomorphism. Then \eqref{o2g} is proved similarly, using $Y$ and $I_{\cC} Y$, and this does not affect \eqref{o1g}.  

Note that we have the expansion
\begin{align}
\begin{split}
\label{o3g}
(\Phi\circ F_2)^*\omega_3 &= \omega_{3}^{\cC}+\frac{1}{\beta}\Big(\frac{1}{\beta}-1\Big)(a_0\Rea+b_0\Ima) (\mathscr{U}^{\frac{1}{\beta}-2} d \mathscr{U}\wedge d \mathscr{\bar V}) + O (r^{\frac{1}{\beta}- 2 - \epsilon})\\
&=\omega_{3}^{\cC}+d(K_{\cC}^*d(-a_0\Ima+b_0\Rea)\mathscr{U}^{\frac{1}{\beta}}) + O (r^{\frac{1}{\beta}- 2 - \epsilon})\\
&=\omega_{3}^{\cC} + (\Phi\circ F_2)^* \Big( 2\sqrt{-1} \partial_K \bar\partial_K (a_0 y - b_0 x) \Big)+ O (r^{\frac{1}{\beta}- 2 - \epsilon}),
\end{split}\end{align}
as $r \to \infty$, where $\epsilon > 0$, $z = x + \i y$ is the $I$-holomorphic function on $X$ asymptotic to $z_{\cC} \circ (\Phi\circ F_2)^{-1} = \mathscr{U}^{\frac{1}{\beta}} \circ (\Phi\circ F_2)^{-1}$, and $\partial_K$ is the $\partial$ operator using the complex structure $K$. Estimates for all higher derivatives analogous to \eqref{o1g}, \eqref{o2g}, and \eqref{o3g} follow from standard arguments using the hyperk\"ahler condition and elliptic regularity.

Next, we have the following key proposition.
\begin{proposition}
\label{t:order-2-equivalence}
If $\beta > 1 / 2$, then for any $\ALG$ gravitational instanton $(X, g, \bm{\omega})$ satisfying
\begin{align}(\Phi\circ F_2)^*\omega_1=\omega_{1}^{\cC} + O (r^{\frac{1}{\beta}-2 - \epsilon}), \quad (\Phi\circ F_2)^*\omega_2=\omega_{2}^{\cC} + O (r^{\frac{1}{\beta}-2 - \epsilon}),\end{align}
and
\begin{align}
(\Phi\circ F_2)^*\omega_3 = \omega_{3}^{\cC} + (\Phi\circ F_2)^* \Big( 2\sqrt{-1} \partial_K \bar\partial_K (a_0 y - b_0 x) \Big) + O (r^{\frac{1}{\beta}- 2 - \epsilon}),
\end{align}
for any $(a,b)\in\mathbb{R}^2$, there exists a smooth function 
\begin{equation}
\varphi = 2 \cdot (a y - a_0 y - b x + b_0 x) + O(r^{\frac{1}{\beta}- \epsilon}),
\end{equation}
as $r \to \infty$, unique up to adding a constant, such that 
$(X, g_{a,b}, \bm{\omega}_{a,b})$ is an $\ALG$ gravitational instanton with the same $\ALG$ coordinate system $\Phi\circ F_2$ and the same $K$, where
\begin{equation}
(\omega_{1,a,b}, \omega_{2,a,b}, \omega_{3,a,b}) = (\omega_1, \omega_2, \omega_{3} + \sqrt{-1} \partial_K\bar\partial_K \varphi),
\end{equation}
and $g_{a,b}$ is defined by $K$ and $\omega_{3,a,b}$.
\end{proposition}
\begin{proof}
We use $K$ as the complex structure and consider
\begin{equation}
\omega_{3, 1} \equiv \omega_3 + 2 \sqrt{-1} \partial_K \bar\partial_K (a y - a_0 y - b x + b_0 x).
\end{equation}
The form $\omega_{3, 1}$ is not positive in general. However, it must be positive on the region $\{r \ge R_1\}$ for sufficiently large $R_1$.
By the surjectivity in \cite[Theorem~4.4]{CCII}, there exists a smooth function $\varphi_1 \in W^{k + 2, 2}_{101 / 100}(\{r \ge R_1\})$ such that
\begin{equation}
2 \omega_{3, 1} \wedge \i \partial_K \bar\partial_K \varphi_1 = \omega_1^2 - \omega_{3, 1}^2,
\end{equation} 
because 
\begin{equation}
\omega_1^2 - \omega_{3, 1}^2 = \omega_3^2 - \omega_{3, 1}^2 = - 4 (\sqrt{-1} \partial_K \bar\partial_K (a y - a_0 y - b x + b_0 x))^2 \in W^{k,2}_{ - 99 / 100}(\{r \ge R_1\})
\end{equation}
for any $k \geq 0$, 
using the fact that $x, y$ are harmonic with respect to the hyperk\"ahler structure $(g, \omega_1, \omega_2, \omega_3)$ and $2 \cdot (\frac{1}{\beta} - 2) \le -1$ by Table~\ref{ALGtable}. Then for large enough $R_2 > R_1$, the form
\begin{equation}
\omega_{3, 2} \equiv \omega_{3, 1} + \i \partial_K \bar\partial_K \varphi_1
\end{equation} 
is smooth and  positive on the region $\{r \ge R_2\}$. Using Lemma~\ref{l:Sobolev}, it follows that 
\begin{align}
(\partial_K \bar\partial_K \varphi_1)^2 \in W^{k,2}_{ - 99 / 50}(\{r \geq R_2\})
\end{align} for any $k \geq 0$. 
So by the surjectivity in \cite[Theorem~4.4]{CCII}, we can find a smooth $\varphi_2 \in W^{k + 2, 2}_{1 / 50}(\{r \ge R_2\})$ such that
\begin{equation}
2 \omega_{3, 2} \wedge \i \partial_K \bar\partial_K \varphi_2 = \omega_1^2 - \omega_{3, 2}^2 = - (\sqrt{-1} \partial_K \bar\partial_K \varphi_1)^2 \in W^{k,2}_{ - 99 / 50}(\{r \geq R_2\})
\end{equation}
for any $k \geq 0$. Then the form
\begin{equation}
\omega_{3, 3} \equiv \omega_{3, 2} + \i \partial_K \bar\partial_K \varphi_2 = \omega_{3, 1} + \i \partial_K \bar\partial_K (2 \cdot (a y - a_0 y - b x + b_0 x) + \varphi_1 + \varphi_2)
\end{equation}
is smooth and positive on the region $\{r \ge R_3\}$ for large enough $R_3 > R_2$. 
 Using Lemma~\ref{l:Sobolev}, it follows that 
\begin{align}
(\partial_K \bar\partial_K \varphi_2)^2 \in W^{k,2}_{ - 99 / 25}(\{r \geq R_3\})
\end{align} for any $k \geq 0$. Then
\begin{equation}
\label{e:decaying}
\omega_1^2 - \omega_{3, 3}^2 = - (\sqrt{-1} \partial_K \bar\partial_K \varphi_2)^2 \in W^{k,2}_{- 99 / 25}(\{r \geq R_3\})
\end{equation}
for any $k \geq 0$. 

Now we take a cut-off function $\chi$ which is $1$ on $(2, \infty)$ and is $0$ on $(-\infty, 1)$. Then as long as $R_4 > R_3$ is large enough, the form
\begin{equation}
\omega_{3, 4} \equiv \omega_{3, 1} + \i \partial_K \bar\partial_K \Big(\chi\Big(\frac{r}{R_4}\Big) \cdot (2 \cdot (a y - a_0 y - b x + b_0 x) + \varphi_1 + \varphi_2) \Big)
\end{equation}
is smooth and positive on the whole manifold $X$. Moreover, by \eqref{e:decaying}, 
\begin{equation}
\int_X (\omega_1^2 - \omega_{3, 4}^2) < \infty.
\end{equation}
Since $\log |z|$ is harmonic with respect to $\omega_3$, the integral
\begin{equation}
\int_X \omega_3 \wedge \i \partial_K \bar\partial_K \Big(\chi\Big(\frac{r}{R_5}\Big) \cdot \log |z|\Big)
\end{equation}
is a non-zero number independent of $R_5$.
Define $t \in \RR$ such that
\begin{equation}
2\cdot t \cdot \int_X \omega_3 \wedge \i \partial_K \bar\partial_K \Big(\chi\Big(\frac{r}{R_5}\Big) \cdot \log |z|\Big) = \int_X (\omega_1^2 - \omega_{3, 4}^2),
\end{equation}
and define
\begin{equation}
\omega_{3, 5} = \omega_{3, 4} + t \i \partial_K \bar\partial_K \Big(\chi\Big(\frac{r}{R_5}\Big) \cdot \log |z|\Big).
\end{equation}
Then $\omega_{3, 5}$ is smooth and still positive for large enough $R_5 > R_4$. Moreover,
\begin{equation}
\int_X (\omega_{3, 5}^2 - \omega_1^2) = \int_X (\omega_{3, 4}^2 - \omega_1^2) + t \cdot \int_X (\omega_{3, 4} + \omega_{3, 5}) \wedge \i \partial_K \bar\partial_K \Big(\chi\Big(\frac{r}{R_5}\Big) \cdot \log |z|\Big).
\end{equation}
Using Stokes' theorem,
\begin{align}
\begin{split}
\int_X \omega_{3, 4} \wedge \i \partial_K \bar\partial_K \Big(\chi\Big(\frac{r}{R_5}\Big) \cdot \log |z|\Big) &= \int_X \omega_{3, 5} \wedge \i \partial_K \bar\partial_K \Big(\chi\Big(\frac{r}{R_5}\Big) \cdot \log |z|\Big) \\
& = \int_X \omega_3 \wedge \i \partial_K \bar\partial_K \Big(\chi\Big(\frac{r}{R_5}\Big) \cdot \log |z|\Big).
\end{split}
\end{align}
Therefore,
\begin{equation}
\int_X (\omega_{3, 5}^2 - \omega_1^2) = 0.
\end{equation}
Finally, we apply Tian-Yau's theorem \cite[Theorem~1.1]{TianYau} to find a bounded smooth function $\varphi_3: X \rightarrow \RR$ such that the form 
\begin{equation}\omega_{3, 6} \equiv \omega_{3, 5} + \i \partial_K \bar\partial_K \varphi_3
\end{equation}
satisfies
\begin{equation}
\omega_{3, 6}^2 = \omega_1^2 =\omega_2^2.
\end{equation}
Hein proved the higher order estimates for $\varphi_3$ \cite[Proposition~2.9]{Hein}. 
Taking $\omega_{3,a,b} = \omega_{3, 6}$ completes the existence proof. 

To show the uniqueness: if $\varphi_{4}$ and $\varphi_{5}$ are two potentials, then
\begin{align}
\begin{split}
&2 \cdot \Big(\omega_3 + \i \partial_K \bar\partial_K \frac{\varphi_{4} + \varphi_{5}}{2}\Big) \wedge \i \partial_K \bar\partial_K  (\varphi_{4}-\varphi_{5}) \\
&= (\omega_3 + \i \partial_K \bar\partial_K \varphi_{4})^2 - (\omega_3 + \i \partial_K \bar\partial_K \varphi_{5})^2 \\
&= \omega_1^2 - \omega_1^2 = 0.
\end{split}
\end{align}
So $\varphi_{4}-\varphi_{5} = O(r^{1 / \beta - \epsilon})$ is harmonic. By \cite[Proposition~7]{HHM}, the leading term of $\varphi_{4}-\varphi_{5}$ is a linear combination of $\log|z|$ and $1$. However, the coefficient of $\log|z|$ must be 0 by Stokes' theorem and the fact that
\begin{equation}
\int_X \Big(\omega_3 + \i \partial_K \bar\partial_K \frac{\varphi_{4} + \varphi_{5}}{2}\Big) \wedge \i \partial_K \bar\partial_K  (\varphi_{4}-\varphi_{5}) = 0.
\end{equation}
Therefore, $\varphi_{4}-\varphi_{5}$ is a constant plus a decaying harmonic function. By the maximum principle, there is no non-zero decaying harmonic function, so $\varphi_{4}-\varphi_{5}$ is a constant.
\end{proof}

Finally, to finish the proof of Theorem \ref{t:order} using Proposition \ref{t:order-2-equivalence}, we apply an argument similar to the one used in \cite[Theorem~A]{CCII} to $\bm{\omega}_{0,0}$ to find another diffeomorphism $F_3: \cC_{\beta,\tau,L}(R) \rightarrow \cC_{\beta,\tau,L}(R)$ homotopic to the identity such that $\bm{\omega}_{0,0}$ is an ALG gravitational instanton with order $2$ in the $\Phi\circ F_3$ coordinates. Actually, the diffeomorphism $F_2^{-1} \circ F_3$ can be chosen to be close to the identity map so that
\begin{equation}
(\Phi\circ F_2)^* \bm{\omega}_{a,b} = (\Phi\circ F_3)^* \bm{\omega}_{a,b} +  O (r^{\frac{1}{\beta} - 2 - \epsilon'})
\end{equation}
for another $\epsilon' > 0$ slightly smaller than $\epsilon > 0$.

\bibliographystyle{amsalpha}

\bibliography{CVZ_II_references}

 \end{document}